\documentclass[11pt]{article}

\usepackage{amsmath, amsfonts, amssymb}
\usepackage{mathrsfs}
\usepackage{theorem}
\usepackage{bm}
\RequirePackage[colorlinks,citecolor=blue,urlcolor=blue,bookmarksopen]{hyperref}

\newtheorem{mythm}{Theorem}[section]

\newtheorem{mylem}[mythm]{Lemma}

{\theorembodyfont{\rmfamily} \newtheorem{myrem}[mythm]{Remark}}
{\theorembodyfont{\rmfamily} \newtheorem{myexam}{Example}[section]}

\newcommand{\ra}{\rightarrow}
\newcommand{\dis}{\displaystyle}

\def\R{\mathbb R}

\def\N{\mathbb N}
\def\C{\mathscr C}

\def\d{\text{\rm{d}}}
\def\E{\mathbb E}
\def\p{\mathbb P}\def\e{\text{\rm{e}}}
\def\q{\mathbb Q}
\def\la{\langle}
\def\raa{\rangle}

\def\veps{\varepsilon}

\def\de{\delta}

\def\C{\mathscr C}

\def\wt{\widetilde}

\newcommand{\fin}{\hfill $\square$\par}
\newenvironment{proof}{{\noindent\it Proof}\quad}{\hfill $\square$\par}

\numberwithin{equation}{section}

\allowdisplaybreaks

\begin{document}

\title{Weak convergence of Euler-Maruyama's approximation for SDEs under integrability condition\footnote{Supported in
 part by    NNSFs of China No. 11771327, 11431014}}

\author{Jinghai Shao\footnote{Email: shaojh@tju.edu.cn}\\[0.2cm]
Center for Applied Mathematics, Tianjin
University, Tianjin 300072, China}
\maketitle

\begin{abstract}
This work  establishes the weak convergence of Euler-Maruyama's approximation for stochastic differential equations (SDEs) with singular drifts under the integrability condition in lieu of the widely used growth condition. This method is based on a skillful application of the dimension-free Harnack inequality. Moreover, when the drifts satisfy certain regularity conditions, the convergence rate is estimated. This method is also applicable when the diffusion coefficients are degenerate. A stochastic damping Hamiltonian system is studied as an illustrative example.
\end{abstract}


\section{Introduction}

Numerical approximation plays important role in the application of stochastic differential equations. According to Hutzenthaler, Jentzen and Kloeden \cite{HJK11,HJK13}, Euler-Maruyama's (EM's) method may diverge to infinity both in the strong and numerically weak sense provided the coefficients of the SDE grow superlinearly. However,  the recently widely studied SDEs with singular coefficients may not satisfy any growth condition or Lyapunov type condition. To define appropriate EM's approximation of such kind of SDEs and prove its weak convergence are the main topic of current work.

SDEs with singular coefficients  have been extensively studied recently. See, for instance, \cite{BC03,GM01,KR05,Zh05,Zh11} for nondegenrate SDEs; see \cite{Ch16,DR02,WZ15,WZ16} for extensions to degenerate SDEs.
In the seminal work \cite{KR05}, Krylov and R\"ockner proved the existence and uniqueness of strong solution to the following SDE:
\begin{equation}\label{1.1}
\d X(t)=b(t,X(t))\d t+\d W(t),\quad X(0)=x_0\in\R^d,
\end{equation}
where $b\in \mathbb L_p^q:=L^q(\R_+;L^p(\R^d))$ with $p,\,q>1$ and $\frac{d}{p}+\frac{2}{q}<1$.
Subsequently, this result was extended to the case of multiplicative noises in \cite{Zh05,Zh11}. See also \cite{FF11,GM01} for related results. Recently, Wang \cite{Wa17} investigated the SDEs with singular coefficients under integrability conditions by using the dimension-free Harnack inequality. This method can deal with the case wherein the singular drift $b$ may not belong to $L_p^q$. For instance, consider the drift
\begin{equation}\label{1.2}
 b(t,x)=b(x)=\Big\{\sum_{n=1}^\infty \log \Big(1+\frac{1}{|x-n|^2}\Big)\Big\}^{\frac 12}- x,\quad x\in\R.
\end{equation} Then SDE \eqref{1.1} admits a unique nonexplosive strong solution (cf. \cite[Theorem 2.1]{Wa17}).
It is obvious that $b(t,x)$ in \eqref{1.2} does not satisfy any growth condition or Lyapunov type condition.

Despite the extensive investigation of SDE \eqref{1.1} with singular coefficients, there are very limited results on the numerical approximation of such systems. In this direction, Zhang
\cite[Theorem 1.1]{Zh16} established a sensitivity  estimation when the drift of such system is perturbed in the space $\mathbb L_p^q$. His method is based on the well-known Krylov's estimate and Zvonkin's transformation; see, for instance, \cite{FF11,KR05,Zh05,Zh11,Zh16} and references therein.
Nevertheless, this method is not applicable to handle directly  EM's  scheme because of the presence of time delay. Intuitively, without time delay, the basic idea to derive Krylov's estimate is that: for $0\leq s<t$,
\begin{align*}
\E\Big[\int_s^tf(r,W(r))\d r\Big|\mathscr F_s\Big]&=\E\Big[\int_s^t f(r,W(s)+W(r)-W(s))\d r\Big|\mathscr F_s\Big]\\
&=\int_s^t \big(2\pi (r-s)\big)^{-\frac d2}\Big(\int_{\R^d} f(r,x+y)\e^{-\frac{|y|^2}{2(r-s)}}\d y\Big)\d r\Big|_{W(s)=x}\\
&\leq N(t-s)^{1-\frac 1q-\frac 2{2p}}\Big(\int_0^\infty\Big(\int_{\R^d}|f(r,y)|^p\d y\Big)^{q/p}\d r\Big)^{1/q}
\end{align*}
for some  $N>0$, where $(W(t))$ is a Brownian motion. Whereas, in the presence of time delay,
\begin{align*}
  \E\Big[\int_s^t f(r,W(\big[\frac{r}{\de}\big]\de))\d r\Big|\mathscr F_s\Big]&=\int_s^t f(r, W(k\de))\d r
\end{align*} for $s,t\in [k\de, (k+1)\de)$, $k\in\N$ and $\de>0$. Thus, it is impossible to get a nonrandom interval function $\rho(s,t)$ such that
\begin{equation*}
\E\Big[\int_s^t f(r,W(\big[\frac{r}{\de}\big]\de))\d r\Big|\mathscr F_s\Big]\leq \rho(s,t)
\end{equation*} for all $0\leq s<t$, so that Krylov's estimate cannot be established without this necessary condition.

Under global Lipschitz condition, it is a classical result that EM's approximation converges strongly to its exact solution (see, e.g. \cite{KP92,Mil95}). Also, under the one-sided Lipschitz condition rather than the global Lipschitz condition, there are also many works about the strong convergence of EM's approximation, which can deal with SDEs with possibly superlinearly growing coefficients (see, e.g.  \cite{HMS02,MS13} and references therein).
Generally, the linear growth condition plays a crucial role in the study of EM's approximation thanks to the observations of  Hutzenthaler et al.   \cite[Theorem 2.1]{HJK11} and
\cite[Theorem 2.1]{HJK13}. There are also many literatures on numerical approximation of SDEs with irregular coefficients. For instance, Yan \cite{Yan02} for SDEs with possibly discontinuous coefficients but satisfying the linear growth condition; Kohatsu-Higa et al. \cite{KMN} for SDEs with bounded H\"older continuous drifts; Ngo and Taguchi \cite{NT} for SDEs with H\"older continuous drifts and satisfying the sub-linear growth condition. If the coefficients are bounded, Kozhina \cite{Koz} performed the sensitivity analysis of the densities of degenerate diffusion processes based on the parametrix expansions of the underlying densities. But the methods of previous mentioned works cannot cope with the SDEs with singular drifts such as \eqref{1.2} or in $\mathbb L_p^q$. According to the characteristics of $b$, it is natural to impose the integrability condition to study the approximation of EM's scheme.

The convergence rate of EM's scheme has been investigated for various convergence criteria: for convergence rate of the expectation of functionals of solutions of SDEs with smooth coefficients, see Talay and Tubaro \cite{TT90}; for convergence rate of the distribution function, see Bally and Talay \cite{BT96}; for convergence rate of the density, see
Bally and Talay \cite{BT96b}, Konakov and Menozzi \cite{KM17}, and Konakov et al. \cite{KKM}; for convergence rate in the Wasserstein distance, see Alfonsi et al. \cite{AJK}. Refer to Kloeden and Platen \cite{KP92} for reviews.


The purpose of this work is to prove the convergence of EM's approximation for SDEs with singular drifts. To this end, there are mainly two issues needing to be addressed. First, in order that EM's approximation of a SDE with singular drift is well-defined, it is necessary to implement certain regularization on the drift such that the drift is well-defined everywhere. As a consequence, one has to measure the difference between these two processes before and after performing regularization.

For example, concerning the drift $b(t,x)$ given in \eqref{1.2},
this drift $b$ is not well defined on the set $\N$. Therefore,
if the initial value of SDE \eqref{1.1} happens to be some $k\in\N$,  its usual EM's approximation cannot be defined. Consequently, one has to complement the definition of $b$. For instance,
one can define a mollification of $b$. To this end, let $\psi\in C^\infty(\R^d;\R_+)$ satisfy $\int_{\R^d}\psi(y)\d y=1$ and $\mathrm{supp}(\psi)\subset K$ for some compact set $K\subset \R^d$. For $\veps>0$, let
\begin{equation}\label{1.3}
\begin{split}
Z(t,x)&=\Big\{\!\sum_{n=1}^\infty \log \Big(1\!+\!\frac{1}{|x-n|^2}\Big)\Big\}^{\frac 12},\quad \psi_{\veps}(x)=\veps^{-d}\psi(x/\veps),\\
b_\veps(t,x)&=\big(Z(t,\cdot)\ast\psi_\veps\big)(x)\!-\! x=\int_{\R^d}Z(t,y)\psi_\veps(x-y)\d y\!-\! x.
\end{split}
\end{equation}
Note that when $b$ is inhomogeneous, we  may need a mollification in time as well.
Associated with $b_\veps$, consider the following SDE:
\begin{equation}\label{1.4}
\d X_\veps(t)=b_\veps(t,X_\veps(t))\d t+\d W(t),\quad X_\veps(0)=x_0.
\end{equation}
Therefore, the first issue is to justify whether $(X_\veps(t))$ converges to $(X(t))$ in appropriate sense. Given  $b$ and $\tilde b$   in $\mathbb L_p^q$ with $p,q>1$ and $\frac{d}{p}+\frac{2}{q}<1$, let $(X^b(t))$, $(X^{\tilde b}(t))$ be the solutions to SDE \eqref{1.1} associated with the drifts $b$ and $b'$ respectively. Then, Zhang \cite{Zh16} showed that
\begin{equation}\label{e-z}\E\big[\sup_{0\leq t\leq T}|X^b(t)-X^{\tilde b}(t)|^2\big]\leq C\|b-\tilde b\|_{\mathbb L_p^q(0,T)}^2
\end{equation}
for some constant $C>0$. This is useful in estimating the convergence of $(X_\veps(t))$ to $(X(t))$ when $b$ belongs to some $\mathbb L_p^q$.  Whereas, as far as the drift \eqref{1.2} is concerned, both this estimate \eqref{e-z} and its method used in \cite{Zh16} do not work any more. Note that the constant $C$ in \eqref{e-z} depends on the norms of $b$ and $\tilde b$ in $\mathbb L_p^q(0,T)$, which restricts the application of \eqref{e-z} to the situation that $b-\tilde b\in \mathbb L_p^q(0,T)$ without assuming $b,\,\tilde{b}\in \mathbb L_p^q(0,T)$. In this work, we shall provide a new estimate of this difference between $(X_\veps(t))$ and $(X(t))$ in terms of the Wasserstein distance.

Second, after the regularization, EM's approximation of $(X_\veps(t))$ is determined by
\begin{equation}\label{1.5}
\d   X_\veps^\de(t)=  b_\veps([t/\de]\de,   X_\veps^\de([t/\de]\de))\d t+\d W(t),\quad  X_\veps^\de(0)=x_0,\ \de>0,
\end{equation} which is well-defined for every initial value $x_0$. The second issue is to study the convergence of $(X_\veps^\de(t))$ to $(X_\veps(t))$ as $\de \ra 0$ and to estimate its convergence rate. During this procedure, attention should be paid to the nonexplosion of $(X^\de_\veps(t))$ without the linear growth condition or one-sided Lipschitz condition.
Herein, we shall study the weak convergence of EM's scheme, i.e. for any bounded measurable function $f$, the convergence of $\E[f(X_\veps^\de(t))]$ to $\E[f(X_\veps(t))]$ as $\de\ra 0$.  This weak convergence and its convergence rate have been studied by many papers and have important application in statistics and mathematical finance.  Moreover, weak convergence can also be characterized by the Wasserstein distance and Fortet-Mourier distance (see, \cite[Chapter 6]{Vi09}). In this work, we would like to present our results in terms of   Fortet-Mourier distance.

Our strategy of this work is to apply the dimension-free Harnack inequality, initiated due to Wang \cite{Wa97}, together with the Girsanov theorem to investigate the weak convergence of EM's approximation.  Recently, Wang \cite{Wa17} used the local Harnack inequality to provide some integrability conditions to ensure the nonexplosion of SDEs with singular drifts and provide more regularity estimates of the invariant probability measures. Subsequently, this method was developed to deal with stochastic partial differential equations with singular and path-dependent drifts in \cite{Wa18}.
Compared with the existing methods of studying weak convergence of EM's scheme, our method owns two advantages. First, in order to ensure the weak convergence, it is sufficient to impose an integrability condition w.r.t. a  nice probability measure in lieu of the growth condition. Second, for SDEs with degenerate diffusion coefficients, this method remains applicable in showing the convergence of EM's scheme under integrability condition.

The remainder of this paper is organized as follows. The main results are presented in Section 2. We consider separately the non-degenerate case and the degenerate case.  All the proofs are presented in Section 3.

\section{Main results}
\subsection{Non-degenerate Case}
Let $(\Omega,\mathcal F,\p)$ be a probability space with a normal filtration $(\mathcal F_t)_{t\geq 0}$ and let $(W(t))_{t\geq 0}$ be a $d$-dimensional standard $\mathcal F_t$-Brwonian motion.
Consider the SDE
\begin{equation}\label{2.1}
\d X(t)=b(t,X(t))\d t+\sigma\d W(t),\qquad X(0)=x_0\in\R^d,
\end{equation} where $(W(t))$ is a $d$-dimensional Brownian motion, and $b:\R_+\times\R^d\ra \R^d$, $\sigma\in \R^{d\times d}$.  The drift $b$ is assumed to be well-defined for every $(t,x)\in\R_+\times\R^d$. $\sigma$ is a deterministic $d\!\times\!d$ matrix satisfying the condition:
\begin{itemize}
   \item[$\mathrm{(H_\sigma)}$] There exists a constant $\lambda>0$ such that $\lambda^{-1}|x|^2\leq |\sigma x|^2\leq \lambda |x|^2$ for all $x\in\R^d$.
 \end{itemize}
Denote $a=\sigma\sigma^\ast=(a_{ij})_{1\leq i,j\leq d}$, where $\sigma^\ast$ stands for the transpose of the matrix $\sigma$. Throughout this work, SDE \eqref{2.1} is assumed to admit a unique weak solution.
Let $(X_\delta(t))$, $\delta>0$, be EM's approximation of $(X(t))$ given by
\begin{equation}\label{2.2}
\d X_\delta(t)=b(t_\delta,X_\de(t_\de))\d t+\sigma \d W(t),\qquad X_\de(0)=x_0,
\end{equation}
where $t_\de=[t/\de]\de$, and $[t/\de]$ denotes the integral part of $t/\de$.

To proceed, we introduce some notation. For $V\in C^2(\R^d)$, define $\mu_0(\d x)=\e^{-V(x)} \d x$ and
\begin{equation}\label{2.3}
 Z_0=-\sum_{i,j=1}^d( a_{ij}\partial_j V) e_i,
\end{equation}
where $\{e_i\}_{i=1}^d$ is the canonical orthonormal basis of $\R^d$ and $\partial_i$ is the directional derivative along $e_i$.
Let
\begin{equation}\label{2.4}
\mathscr V=\Big\{V\!\in \!C^2(\R^d);\ \mu_0(\R^d)=1, \begin{array}{l}\exists\,K_0>0, \,|Z_0(x)\!-\!Z_0(y)| \leq K_0|x\!-\!y|^2\ \ \forall\,x,y\in\R^d   
 \end{array}
\Big\}.
\end{equation}
The Lipschitz condition in \eqref{2.4} is used to establish the Harnack inequality. For $V\in\mathscr V$, define
\begin{equation}\label{2.5}
L_0=\mathrm{tr}(a\nabla^2)+Z_0\cdot \nabla=\sum_{i,j=1}^da_{ij}\partial_i\partial_j+\sum_{i=1}^d\la Z_0,e_i\raa \partial_i.
\end{equation}
By the integration by parts formula, $L_0$ is symmetric in $L^2(\mu_0)$. Then
\[\mathscr E_0(f,g):=\mu_0\big(\la\nabla f,\nabla g\raa\big),\qquad f,g\in H_\sigma^{2,1}(\mu_0)\]
is a symmetric Dirichlet form generated by $L_0$, where $H_{\sigma}^{2,1}$ is the closure of $C_0^{\infty}(\R^d)$ under the following norm
\[\|f\|_{H_\sigma^{2,1}(\mu_0)}:=\big\{\mu_0(|f|^2+|\sigma^\ast \nabla f|^2)\big\}^{\frac 12}.\]
Also, for each $V\in \mathscr V$, it is associated with an auxiliary process $(Y(t))$ determined by
\begin{equation}\label{2.6}
\d Y(t)=Z_0(Y(t))\d t+\sigma \d W(t),\quad Y(0)=x_0.
\end{equation} For $V\in\mathscr V$, $Z_0$ is globally Lipschitz continuous, hence,  $(Y(t))$ is nonexplosive (cf. e.g. \cite{Mao}).

The Wasserstein distance between two probability measures  $\mu,\,\nu$  on $\R^d$ is defined by
\[W_1(\mu,\nu)=\inf\Big\{\int_{\R^d\times\R^d} |x-y|\pi(\d x,\d y);\ \pi\in\C(\mu,\nu)\Big\},
\]where $\C(\mu,\nu)$ denotes the collection of all probability measures $\pi$ on $\R^d\times\R^d$ such that $\pi(A\times \R^d)=\mu(A)$ and $\pi(\R^d\times B)=\nu(B)$ for all Borel measurable sets $A$, $B$ on $\R^d$.
According to the Kantorovich-Rubinstein Theorem (cf. \cite[Theorem 1.14]{Vi}),
\begin{equation}\label{wass}
\begin{aligned}
  W_1(\mu,\nu)=\sup\Big\{\int_{\R^d}\phi\, \d \mu-\int_{\R^d}\phi\,\d\nu;\ \|\phi\|_{\mathrm{Lip}}\leq 1,\,\phi \ \text{is bounded}\Big\},
\end{aligned}
\end{equation} where $\|\phi\|_{\mathrm{Lip}}=\sup_{x,y,\in \R^d,x\neq y}\frac{|\phi(x)-\phi(y)|}{|x-y|}$. For a sequence of probability measures $(\mu_n)$ on $\R^d$, $\mu_n$ converges in the Wasserstein distance $W_1$ to some probability measure $\mu$ on $\R^d$ is equivalent to that $\mu_n$ weakly converges to $\mu$ and $\int_{\R^d}|x|\mu_n(\d x)$ converges to $\int_{\R^d}|x|\mu(\d x)$. See \cite{Vi} or \cite{RR} for more researches on the Wasserstein distance. We denote by $\mathscr{L}(\zeta)$ the distribution of a random variable $\zeta$, and by $\mu(f)$ the integration of function $f$ w.r.t. the measure $\mu$, i.e. $\mu(f)=\int_{\R^d}f(x)\mu(\d x)$. A closely related distance is the Fortet-Mourier distance (also called bounded Lipschitz distance):
\begin{equation}\label{bL-dis}
\begin{aligned}
W_{bL}(\mu,\nu)=\sup\Big\{\int_{\R^d}\phi\,\d \mu-\int_{\R^d}\phi\,\d\nu;\ \|\phi\|_{\mathrm{Lip}}+\|\phi\|_\infty\leq 1\Big\}.
\end{aligned}
\end{equation}
The Fortet-Mourier distance can also characterize the weak convergence of the probability measure space (cf. \cite[Chapter 6]{Vi09a}).

\begin{mythm}\label{t-0}
Let $(X(t))$ be the solution to SDE \eqref{2.1} and  $(\wt X (t))$ be the solution to the following SDE:
\begin{equation}\label{u-1}
\d \wt X(t)=\tilde b(t,\wt X(t))\d t+\sigma \d W(t),\quad \wt X(0)=x_0,
\end{equation}
where $\tilde b:\R_+\times\R^d\ra \R^d$.
Suppose condition $(H_\sigma)$ holds.  Let $T>0$ be given. Assume there exists $V\in\mathscr V$ such that its associated vector $Z_0$ determined by \eqref{2.3}, $Z(t,x):=b(t,x)-Z_0(x)$ and $\wt Z(t,x):=\tilde b(t,x)-Z_0(x)$ satisfy the following condition:
\begin{itemize}
  \item[$\mathrm{(H1)}$] There exists  a constant  $\eta>2\lambda Td$ such that
      \[ \sup_{t\in[0,T]}\mu_0\big(\e^{\eta |Z(t,\cdot)|^2}\big)<\infty, \  \text{and}\ \sup_{t\in[0,T]}\mu_0\big(\e^{\eta |\wt Z(t,\cdot)|^2}\big)<\infty.\]
\end{itemize}
Then, for every $\xi>d$, there exists a constant $C=C(K_0,T,\lambda,\xi,\eta)$ such that
  \begin{equation}\label{u-2}
  \begin{split}
   \sup_{t\in [0,T]} W_{bL}(\mathscr L(X(t)),\mathscr L(\wt X(t)))
  \leq C\Big\{\int_0^T\!\frac{\mu_0(|Z-\wt Z|^{q_0\xi}(s,\cdot))^{\frac 1\xi}}{(1-\e^{-K_0s})^{\frac{d}{\xi}}}\d s\Big\}^{\frac 1{q_0}},
  \end{split}
  \end{equation} where $q_0=p_0/(p_0-1)$, $p_0=\sqrt{\frac{\eta}{2\lambda Td}}\wedge 2$.
\end{mythm}
\begin{myrem} According to \cite[Theorem 2.1]{Wa17}, SDEs \eqref{2.1} and \eqref{u-1} admit a unique solution under the condition (H1).
The right hand side of \eqref{u-2} can be viewed as a weighted norm of $Z-\wt Z$ on $[0,\infty)\times\R^d$ w.r.t. the finite measure $\vartheta_\xi(t)\d t\times \mu_0(\d x)$, where
$\vartheta_\xi(t)=\frac{1}{(1-\e^{-K_0t})^{\frac{d}{\xi}}}$ is an integrable function on $[0,\infty)$ for $\xi>d$.
\end{myrem}

\begin{mythm}\label{t-1}
Assume $\mathrm{(H_\sigma)}$ holds. Let $T>0$ be given. Assume there exists $V\in\mathscr V$ such that $(t,x)\mapsto Z(t,x):=b(t,x)-Z_0(x)$ is continuous, where $Z_0$ is associated with $V$  by \eqref{2.3}.
Suppose
\begin{itemize}
  \item[$\mathrm{(H2)}$] there exist constants $\eta_0>0$ and $\eta>4\lambda Td$ such that
      \[\mu_0\big(\e^{ \eta_0 |Z_0|^2}\big)<\infty\quad \text{and}\quad \sup_{t\in[0,T]}\mu_0\big(\e^{\eta |Z(t,\cdot)|^2}\big)<\infty.\]
\end{itemize}
Then, for any bounded measurable function $f$ on $\R^d$, it holds
\begin{equation}\label{2.7}
\lim_{\de\ra 0}\E[f(X_\de(t))]=\E[f(X(t))],\quad t\in [0,T],
\end{equation}
and
\begin{equation}\label{2.7.5}
\lim_{\de\ra 0}W_{bL}(\mathscr L(X_\de(t)), \mathscr L(X(t)))=0,\quad t\in[0,T].
\end{equation}
\end{mythm}

In order to estimate the weak convergence rate of above EM's scheme, certain regularity conditions on the drift $b(t,x)$ are needed. A measurable function $h:\R^d\ra \R$ is said to be \emph{polynomially bounded} if there exist positive constants $K,\,p$ such that $|h(x)|\leq K(1+|x|^p)$ for all $x\in\R^d$.
\begin{mythm}\label{t-2}
Suppose the conditions in Theorem \ref{t-1} hold. In addition, assume that  there exist constants $K_1,\,m_1>0, \alpha\in(0,1]$ and a polynomially bounded function $h:\R^d\ra \R_+$ such that
\begin{equation}\label{2.8}
\begin{split}
|Z(t,x)-Z(t,y)|&\leq K_1(1+|x|^{m_1}+|y|^{m_1})|x-y|, \quad  t\in[0,T],\ x,y\in\R^d,\\
|Z(t,x)-Z(s,x)|&\leq h(x) |t-s|^{\alpha},\qquad  t,s\in[0,T],\ x\in\R^d.
\end{split}
\end{equation}
Then, for any bounded measurable function $f$ on $\R^d$,
\begin{equation}\label{2.9}
|\E[f(X_\de(t))]-\E[f(X(t))]|\leq C \de^{\frac 12\wedge \alpha},\quad t\in[0,T],
\end{equation} and
\begin{equation}\label{2.9.5} W_{bL}(\mathscr L(X_\de(t)),\mathscr L(X(t)))\leq C \de^{\frac 12\wedge \alpha},\quad t\in[0,T],
\end{equation}
where $C$ is a constant depending on $\lambda,\, d, m_1,\alpha, K_0, K_1, \eta$ and $\sup_{t\leq T} \E[ h(Y(t))]$. Because of the linear growth of $Z_0$ and the polynomial boundedness of $h$,  $\sup_{t\leq T} \E[ h(Y(t))]$ is finite.
\end{mythm}

\begin{myexam}
Let $b$ and $b_\veps$ be given by \eqref{1.2} and \eqref{1.3}. In addition, without loss of the generality, assume $\mathrm{supp}(\psi)\subset [-1,1]$. Then, take
$\dis \mu_0(\d x)=\frac{\e^{-x^2/2}}{\sqrt{2\pi}}\d x$.
Correspondingly,
\[Z_0(x)=-x,\quad Z(x)=\Big\{\sum_{n=1}^\infty \log\Big(1+\frac{1}{|x-n|^2}\Big)\Big\}^{\frac 12}.\]
For any $\eta>0$, it holds
\begin{equation}\label{ex-1}
\mu_0\Big(\e^{\eta |Z|^2}\Big)<\infty.
\end{equation}
Precisely, it is sufficient to show the finiteness of this integral over the positive half line. To this end,
for $k\in\N$, it holds
\begin{align*}
  &\int_{k+\frac 12}^{k+\frac 32} \e^{\eta \sum_{n=1}^\infty \log\big(1+\frac{1}{|x-n|^2}\big)}\mu_0(\d x)\\
  &\leq \int_{k+\frac 12}^{k+\frac 32} \e^{\eta \log\big(1+\frac1{|x-k-1|^2}\big)}\e^{\eta \sum_{\ell=0}^{k-1}\log\big(1+\frac{1}{(\frac1 2+\ell)^2}\big)} \e^{\eta \sum_{\ell=0}^\infty \log\big(1+\frac{1}{(\frac 12+\ell)^2}\big)}\mu_0(\d x)\\
  &\leq  \e^{2\eta \sum_{\ell=0}^\infty \log\big(1+\frac{1}{(\frac 12+\ell)^2}\big)}\int_{k+\frac 12}^{k+\frac 32} \big(1+\frac{1}{|x-k-1|^2}\big)^\eta\mu_0(\d x)\\
  &\leq \e^{2\eta \sum_{\ell=0}^\infty \log\big(1+\frac{1}{(\frac 12+\ell)^2}\big)}\frac{\e^{- (k+\frac 12)^2/2}}{\sqrt{2\pi}}\Big(\int_{-\frac 12}^{\frac 12} \big(1+\frac{1}{|x|^2}\big)^\eta\d x\Big).
\end{align*}
Then
\begin{align*}
  &\int_{\frac 12}^\infty\!\!\e^{\eta\sum_{n=1}^\infty \log\big(1+\frac{1}{|x-n|^2}\big)}\mu_0(\d x)= \sum_{k=0}^\infty\int_{k+\frac 12}^{k+\frac{3}{2}} \e^{\eta\sum_{n=1}^\infty\log \big(1+\frac{1}{|x-n|^2}\big)}\mu_0(\d x)\\
  &\leq  \e^{2\eta \sum_{\ell=0}^\infty \log\big(1+\frac{1}{(\frac 12+\ell)^2}\big)}\Big(\int_{-\frac 12}^{\frac 12} \big(1+\frac{1}{|x|^2}\big)^\eta\d x\Big)\sum_{k=0}^\infty \frac{\e^{- (k+\frac 12)^2/2}}{\sqrt{2\pi}}  <\infty.
\end{align*}Hence \eqref{ex-1} holds. Similarly, one can show that
\[Z_\veps(x)=\big(Z\ast \psi_\veps\big)(x)\quad \text{and}\quad \frac{\d Z_\veps(x)}{\d x} =\int_{\R}Z(y)\frac{\d }{\d x}\psi_\veps(x-y)\d y, \quad x\in \R^d, \]
are all bounded.

Let $(X_\veps(t))$ and $(X_\veps^\de(t))$ be determined by \eqref{1.4} and \eqref{1.5} respectively. Then, according to Theorem \ref{t-0} and Theorem \ref{t-2}, for every given $T>0$, $\xi>1$, there exists a constant $C$ such that
\begin{align*}
   W_{bL}(\mathscr L(X_\veps(t)), \mathscr L(X(t)))\leq C\big(\mu_0(|Z-Z_\veps|^{2\xi})\big)^{\frac 1{2\xi}}, \quad t\in[0,T],
\end{align*}
and
\begin{align*}
  W_{bL}(\mathscr L(X_\veps^\de(t)),\mathscr L(X_\veps(t)))\leq C \de^{\frac 12},\quad t\in[0,T].
\end{align*}
\end{myexam}

\subsection{Degenerate Case}

It is known that uniformly elliptic condition plays an important role in the study of EM's scheme. As shown by Hairer et al. \cite[Theorem 5.1]{HHJ}, an example SDE with globally bounded and smooth coeffiecients was constructed to show that the standard EM's approximation converges to the exact solution of this SDE in the strong and weak sense, but at a rate that is slower than any power law.
Whereas, our method used in the previous subsection can be extended without much additional difficulties to the case of  degenerate  SDEs in the form:
\begin{equation}\label{deg-1}
\begin{aligned}
  \d X^{(1)}(t)&= X^{(2)}(t)\d t,\\
  \d X^{(2)}(t)&=b(t,X^{(1)}(t),X^{(2)}(t))\d t+\sigma \d W(t).
\end{aligned}
\end{equation}
This system is known as a stochastic damping Hamiltonian system.
We shall use $(X^{(1)}(t,x_1)$, $ X^{(2)}(t,x_2))$ to denote the solution to \eqref{deg-1} with initial value $(x_1,x_2)\in \R^{2d}$.
For any $f\in \mathscr B_b(\R^{2d})$, the set of all bounded measurable functions on $\R^{2d}$, let
\[P_tf(x_1,x_2)=\E f\big(X^{(1)}(t,x_1), X^{(2)}(t,x_2)\big). \]
Then $u(t,x_1,x_2):=P_t f(x_1,x_2)$ solves a degenerated Fokker-Planck type equation, which has recently attracted much attention in the name ``kinetic Fokker-Planck equation"; see Villani \cite{Vi09}. The long time behavior of $P_t$ has been investigated in \cite{BCG,Wu01}. Guillin and Wang in \cite{GW12} established the Bismut formula and Harnack inequality of this system. The sensitivity analysis of the densities of such system in time-homogeneous case w.r.t. a perturbation of the coefficients of non-degenerate component has been performed by Kozhina \cite{Koz} following \cite{KKM}. Their proofs are based on the parametrix expansions of the underlying densities, which demands, in particular, the boundedness of the coefficients.

The results in the previous subsection all have their corresponding extension to the degenerate system \eqref{deg-1}. In what follows,    the corresponding extension of Theorems \ref{t-0}, \ref{t-1} and \ref{t-2} are presented.

Let $(\wt X^{(1)}(t), \wt X^{(2)}(t))$ be the solution to the following SDE:
\begin{equation}\label{deg-2}
\begin{split}
  \d \wt X^{(1)}(t)&=\wt X^{(2)}(t)\d t,\\
  \d \wt X^{(2)}(t)&=\tilde b(t, \wt X^{(1)}(t),\wt X^{(2)}(t))\d t +\sigma \d W(t),
\end{split}
\end{equation} where $\tilde b:\R_+\times \R^{2d}\ra \R^d$. Throughout this subsection, we assume that SDEs \eqref{deg-1} and \eqref{deg-2} admit unique weak solutions.

\begin{mythm}\label{t-deg-1}
Let $T>0$ be given. Suppose that $\mathrm{(H_\sigma)}$ holds and there exists  $V\in \mathscr V$ such that $Z_0$, $Z(t,x_1,x_2):=b(t,x_1,x_2)-Z_0(x_2)$, $\wt Z(t,x_1,x_2):=\tilde b(t,x_1,x_2)-Z_0(x_2)$ satisfy:
\begin{itemize}
  \item[$\mathrm{(A1)}$] there exists $\eta>2\lambda T d$ such that
  \begin{equation}\label{deg-3}
  \sup_{t\in[0,T],x_1\in\R^d}\mu_0\Big(\e^{\eta |Z(t,x_1,\cdot)|^2}\Big)<\infty,\quad \sup_{t\in[0,T],x_1\in\R^d}\mu_0\Big(\e^{\eta|\wt Z(t,x_1,\cdot)|^2}\Big)<\infty.
  \end{equation}
\end{itemize}Then, it holds for every $\xi>d$ that
\begin{equation}\label{deg-4}
\begin{split}
&\sup_{t\in[0,T]}W_{bL}\big(\mathscr L(X^{(1)}(t),X^{(2)}(t)), \mathscr L(\wt X^{(1)}(t),\wt X^{(2)}(t))\big)\\
&\leq C\Big\{\int_0^T \frac{\mu_0\big(\sup_{x_1\in\R^d}|Z(s,x_1,\cdot)-\wt Z(s,x_1,\cdot)|^{q_0\xi}\big)^{\frac 1\xi}}{(1-\e^{-K_0s})^{\frac{d}{\xi}}}\,\d s\Big\}^{\frac{1}{q_0}},
\end{split}
\end{equation}
where $q_0=p_0/(p_0-1)$, $p_0=\sqrt{\frac{\eta}{2\lambda Td}}\wedge 2$.
\end{mythm}

Next, let us consider EM's approximation of SDE \eqref{deg-1}, which is determined by
\begin{equation}\label{deg-5}
\begin{split}
  \d X^{(1)}_\de(t)&=X^{(2)}_\de(t)\d t,\\
  \d X^{(2)}_\de(t)&=b(t_\de, X^{(1)}_\de(t_\de),X^{(2)}_\de(t_\de))\d t+\sigma \d W(t),
\end{split}
\end{equation} where $t_\de=[t/\de]\de$ for $\de>0$.

\begin{mythm}\label{t-deg-2}
Let $T>0$ be given. Assume that $\mathrm{(H_\sigma)}$ holds and there exists $V\in\mathscr V$ such that $(t,x_1,x_2)\mapsto Z(t,x_1,x_2):=b(t,x_1,x_2)-Z_0(x_2)$ is continuous. Suppose that
\begin{itemize}
  \item[$\mathrm{(A2)}$] there exist constants $\eta_0>0$, $\eta>4\lambda T d$ such that
  \[\mu_0\big(\e^{\eta_0|Z_0|^2}\big)<\infty \quad \text{and}\quad \sup_{t\in[0,T]}\mu_0\big(\sup_{x_1\in \R^d}\e^{\eta |Z(t,x_1,\cdot)|^2}\big)<\infty.\]
\end{itemize}
Then, for any bounded measurable function $f$ on $\R^{2d}$,
\begin{equation}\label{deg-6}
\lim_{\de\ra 0}\big|\E f(X_\de^{(1)}(t),X_\de^{(2)}(t))-\E f(X^{(1)}(t),X^{(2)}(t))\big|=0,
\end{equation}
and
\begin{equation}\label{deg-7}
\lim_{\de\ra 0}W_{bL}\big(\mathscr L(X_\de^{(1)}(t),X_\de^{(2)}(t)),\mathscr L(X^{(1)}(t),X^{(2)}(t))\big)=0.
\end{equation}
\end{mythm}

\begin{mythm}\label{t-deg-3}
Assume the conditions of Theorem \ref{t-deg-2} are fulfilled, and  further assume that there exist $K_2,\,m_2>0$, $\alpha\in (0,1]$ and a polynomial bounded function $g:\R^{2d}\ra \R_+$ such that
\begin{equation}\label{deg-8}
\begin{split}
  |Z(t,x_1,x_2)-Z(t,\,y_1,y_2)|&\leq K_2\big(1\!+\!|x_1|^{m_2}\!+\!|x_2|^{m_2}\!+\!|y_1|^{m_2}\!+\!|y_2|^{m_2}\big)|x_2-y_2|,\\
  |Z(t,x_1,x_2)-Z(s,x_1,x_2)|&\leq g(x_1,x_2)|t-s|^\alpha
\end{split}
\end{equation} for $t,\,s\in [0,T]$, $x_1,x_2,y_1,y_2\in\R^d$.
Then, for any bounded measurable function $f$ on $\R^{2d}$,
\begin{equation}\label{deg-9}
\big|\E f(X_\de^{(1)}(t),X_\de^{(2)}(t))-\E f(X^{(1)}(t),X^{(2)}(t))\big|\leq C \de^{\frac 12\wedge \alpha},
\end{equation}
and
\begin{equation}\label{deg-10}
W_{bL}\big(\mathscr L(X^{(1)}_\de(t),X_\de^{(2)}(t)),\mathscr L(X^{(1)}(t),X^{(2)}(t))\big)\leq C\de^{\frac 12\wedge \alpha}
\end{equation} for some positive constant $C$.
\end{mythm}

\section{Proofs of the main results}

\subsection{Non-degenerate Case}
The basic idea to prove the main results is to construct an auxiliary process which provides a new representation of   $(X(t))$ or its EM's approximation $(X_\de(t))$ based on the Girsanov theorem. Then the key point is to apply the Harnack inequality to verify Novikov's condition so that this kind of transformation is well-defined.   Precisely, for $V\in\mathscr V$, its associated process $(Y(t))$ is defined by \eqref{2.6}. The global Lipschitz condition of $Z_0$ ensures that $(Y(t))$ is nonexplosive. Let $T>0$ be given. Let us rewrite $(Y(t))$ into the following form:
\begin{equation}\label{rep-X}
\begin{split}
\d Y(t)&=b(t,Y(t))\d t +\sigma\big(\d W(t)-\sigma^{-1} b(t, Y(t))\d t+\sigma^{-1}Z_0(Y(t))\d t\big)\\
&=b(t,Y(t))\d t+\sigma\d \widehat W_1(t),
\end{split}
\end{equation}
where
\begin{equation}\label{hat-W}
\widehat W_1(t)=W(t)+\int_0^t\sigma^{-1}\big(Z_0(Y(s))-b(s,Y(s))\big)\d s=W(t)-\int_0^t\sigma^{-1}Z(s,Y(s))\d s.
\end{equation}
If Novikov's condition
\begin{equation}\label{Nov-1}
\E\exp\Big[\frac 12\int_0^T|\sigma^{-1}(Z(s,Y(s)))|^2\d s\Big]<\infty
\end{equation}
holds, then
\begin{equation}\label{q1}
\begin{aligned}
\mathbb Q_1&:=\exp\Big[\int_0^T\la \sigma^{-1}(Z(s,Y(s))),\d W(s)\raa-\frac 12\int_0^T|\sigma^{-1}(Z(s,Y(s)))|^2\d s\Big]\p
\end{aligned}
\end{equation}
is a probability measure. Applying the Girsanov theorem,  $(\widehat W_1(t))$ is a new Brownian motion under the probability $\q_1$. Moreover, $(Y(t),\widehat W_1(t))$ under $\q_1$ is also a solution of \eqref{2.1}. By the weak uniqueness of the solution to  \eqref{2.1}, $(Y(t))$ under $\q_1$ has the same distribution as that of $(X(t))$ under $\p$.

In order to deal with  EM's approximation $(X_\de(t))$, rewrite \eqref{2.6} into the following form:
\begin{equation}\label{rep-EM}
\d Y(t)=b(t_\de,Y(t_\de))\d t+\sigma \d \widehat W_2(t),
\end{equation}
where
\begin{equation}\label{tilde-W}
\widehat W_2(t)=W(t)+\int_0^t\sigma^{-1} (Z_0(Y(s))-b(s_\de,Y(s_\de)))\d s,\ t\in [0,T].
\end{equation}
If
\begin{equation}\label{Nov-2}
\E\exp\Big[\frac 12\int_0^T|\sigma^{-1}(Z_0(Y(s))-b(s_\de,Y(s_\de)))|^2\d s\Big]<\infty
\end{equation}
holds, then
\begin{equation}\label{q2}
\begin{aligned}
\mathbb Q_2&:=\exp\Big[-\int_0^T\la \sigma^{-1}(Z_0(Y(s))-b(s_\de,Y(s_\de))),\d W(s)\raa\\
&\quad-\frac 12\int_0^T|\sigma^{-1}(Z_0(Y(s))-b(s_\de,Y(s_\de)))|^2\d s\Big]\p
\end{aligned}
\end{equation}
is a probability measure. Hence, the Girsanov theorem yields that $(\widehat W_2(t))$ is a new Brownian motion under the probability $\q_2$. $(Y(t),\widehat W_2(t))$ under $\q_2$ is a solution of \eqref{2.2}. By the uniqueness of solution for \eqref{2.2}, $(Y(t))$ under $\q_2$ has the same distribution as that of $(X_\de(t))$ under $\p$.

Let us first make some preparations before proving the main results.

\begin{mylem}\label{t-3.1}
Let $G:\R_+\times\R^d\ra \R_+$ be a measurable function and $\beta>0$ be a constant.
\begin{itemize}
\item[(i)]
If there exists a constant $\xi>d$ such that $\sup_{t\in[0,T]}\mu_0(G^{\xi}(t,\cdot))<\infty$, then
\begin{equation}\label{3.1.5}
\E\Big[\int_0^T G(s,Y(s))\d s\Big]\leq C \sup_{t\in[0,T]}\mu_0(G^{\xi}(t,\cdot))^{\frac 1\xi}<\infty
\end{equation} for some constant $C=C(T,\xi,K_0)>0$.
\item[(ii)] If there exists a constant $\eta$ such that $\eta>\beta T d$ and
$\dis \sup_{t\in[0,T]}\mu_0\big(\e^{\eta G(t,\cdot)}\big)<\infty$,
then
\begin{equation}\label{3.1}
\E\Big[\e^{\beta\int_0^T G(s,Y(s))\d s}\Big]<\infty, \quad \text{and}\ \quad \E\Big[\e^{\beta\int_0^T G(s,Y(s_\de))\d s}\Big]<\infty.
\end{equation}
\end{itemize}
\end{mylem}

\begin{proof} We first prove the assertion (ii), and then (i) follows immediately.
  Let $P_t^0$ denote the semigroup corresponding to the process $(Y(t))$ defined by \eqref{2.6} with initial value $Y(0)=x$. Hence, the semigroup $P_t^0$ is symmetric w.r.t. $\mu_0$. Since $V\in\mathscr V$, according to
  \cite[Theorem 1.1]{Wa11}, for $p>1$, the following Harnack inequality holds:
  \begin{equation}\label{3.2}
    \Big(P_t^0f(x)\Big)^p\leq P_t^0f^p(y)\exp\Big[\frac{K_0\sqrt{p}}{\sqrt{p}-1}\cdot\frac{|x-y|^2}{1-\e^{-K_0t}}\Big],\quad \forall\,f\in \mathscr B_b^+(\R^d).
  \end{equation}
  Applying the Harnack inequality \eqref{3.2}, we get for $\gamma>0$ and $N>0$
  \begin{align*}&\Big\{\E \e^{\gamma G(t,Y(t))\wedge N}\Big\}^p=\Big\{P_t^0\e^{\gamma G(t,\cdot)\wedge N}\Big\}^p(x)\\
  &\leq \Big\{P_t^0\e^{\gamma p G(t,\cdot)\wedge N}\Big\}(y)\exp\Big[\frac{K_0\sqrt{p}}{\sqrt{p}-1}\cdot\frac{|x-y|^2}{1-\e^{-K_0t}}\Big].
  \end{align*}
  Passing to the limit as $N\ra +\infty$, it follows from Fatou's lemma that
  \begin{equation}\label{3.3}
  \Big\{P_t^0\e^{\gamma G(t,\cdot)}\Big\}^p(x)\leq  \Big\{P_t^0\e^{\gamma pG(t,\cdot)}\Big\}(y)\exp\Big[\frac{K_0\sqrt{p}}{\sqrt{p}-1}\cdot\frac{|x-y|^2}{1-\e^{-K_0t}}\Big].
  \end{equation}
  Denote $B(x,r)=\{y\in\R^d; |y-x|\leq r\}$ for $r>0$, $x\in\R^d$. Integrating both sides of \eqref{3.3} w.r.t. $\mu_0$ over the set $B(x, \sqrt{1-\e^{-K_0t}})$, we obtain
  \begin{equation}\label{3.4}
  \begin{split}
  &\Big\{P_t^0\e^{\gamma G(t,\cdot)}(x)\Big\}^{p}\mu_0\big(B\big(x,\sqrt{1-\e^{-K_0t}}\big)\big)\\
  &\leq \int_{B(x,\sqrt{1-\e^{-K_0t}})}\big\{P_t^0\e^{\gamma p G(t,\cdot)}\big\}(y)\e^{\frac{K_0 \sqrt{p}}{\sqrt{p}-1} \cdot \frac{|x-y|^2}{1-\e^{-K_0t}}}\mu_0(\d y)\\
  &\leq \int_{B(x,\sqrt{1-\e^{-K_0t}})}\big\{P_t^0\e^{\gamma p G(t,\cdot)}\big\}(y)\e^{\frac{K_0\sqrt{p}}{\sqrt{p}-1} }\mu_0(\d y)\\
  &\leq \e^{\frac{K_0\sqrt{p}}{\sqrt{p}-1} }\mu_0(\e^{\gamma p G(t,\cdot)}).
  \end{split}
  \end{equation}
  Since $\mu_0$ has strictly positive and continuous density $\e^{-V}$ w.r.t. the Lebesgue measure, there exists $\Gamma\in C(\R^d;(0,\infty))$ such that $\mu_0(B(x,t))\geq \Gamma(x) t^d$ for $t\in (0,1]$ and $x\in\R^d$. Invoking \eqref{3.4}, we obtain
  \begin{equation}\label{3.5}
  \E\e^{\gamma G(t,Y(t))}\leq \Gamma(x)^{-\frac 1p}\e^{\frac{K_0}{p-\sqrt p}}\mu_0\Big(\e^{\gamma p G(t,\cdot)}\Big)^{\frac 1p}\frac{1}{\big(1-\e^{-K_0t}\big)^{  d/p}}, \ \quad t\in(0,T].
  \end{equation}
  Combining this with Jensen's inequality, we get
  \begin{equation}\label{3.6}
  \begin{aligned}
    \E\Big[\e^{\beta\int_0^TG(t,Y(t))\d t}\Big]&\leq \frac 1T\int_0^T \E\Big[\e^{\beta TG(t,Y(t))}\Big]\d t\\
    &\leq \frac{C}{\Gamma(x)^{1/p}}\sup_{t\in[0,T]}\mu_0\Big(\e^{\beta T pG(t,\cdot)}\Big)^{1/p}\int_0^T\frac{1}{(1-\e^{-K_0t})^{d/p}} \d t,
  \end{aligned}
  \end{equation} where $C=C(p,T,K_0)$ is a constant.
  Taking $d<p<\frac{\eta}{\beta T }$ in \eqref{3.6},   it follows from the assumed condition in (ii) that
  \[\E\Big[\e^{\beta\int_0^TG(t,Y(t))\d t}\Big]<\infty.\]
  Similarly, by taking $d<p<\frac{\eta}{\beta T}$ and using the Harnack inequality \eqref{3.2}, we can deduce that
   \begin{equation}\label{3.7}
   \begin{aligned}
   &\E\Big[\e^{\beta\int_0^T G(s,Y(s_\de))\d s}\Big]\\
   &=\E\Big[\e^{\beta\int_0^\de G(s,Y(s_\de))\d s+\beta\int_\de^T G(s,Y(s_\de))\d s}\Big]\\
   &\leq \frac{\e^{\beta\int_0^\de G(s,x)\d s} }{T-\de}\int_\de^T \E\Big[\e^{\beta(T-\de) G(Y(s_\de))}\Big]\d s\\
   &\leq \frac{\e^{\frac{K_0}{p-\sqrt p} }}{ T-\de}\cdot \frac{\e^{\beta\int_0^\de G(s,x)\d s}}{\Gamma(x)^{1/p}}\cdot \sup_{t\in[\de,T]}\!\mu_0\big(\e^{\beta(T-\de)p G(t,\cdot)}\big)^{1/p}\int_\de^T \!\!\frac{1}{(1\!-\!\e^{-K_0 s_\de})^{d/p}}\d s\\
   &\leq C \sup_{t\in[0,T]}\!\mu_0\big(\e^{\beta(T-\de)p G(t,\cdot)}\big)^{1/p}\int_0^T \!\!\frac{1}{(1-\e^{-K_0 s})^{d/p}}\d s<\infty,
   \end{aligned}
   \end{equation} where $C=C(p,T,K_0)$ is a constant.

   In order to establish \eqref{3.1.5}, noticing $\xi>d$,  we obtain from \eqref{3.5} that
   \begin{equation}\label{n-1}
     \E[G(s,Y(s))]\leq \frac{\e^{\frac{K_0}{\xi-\sqrt{\xi}} }\mu_0(G^{\xi}(s,\cdot))^{\frac 1\xi}}
     {\Gamma(x)^{\frac{1}{\xi}}(1-\e^{-K_0s})^{\frac{d}{\xi}}},\quad s\in (0,T],
   \end{equation}
   and hence
   \begin{align*}
     \E\Big[\int_0^T G(s,Y(s))\d s\Big]\leq \frac{\e^{\frac{K_0}{\xi-\sqrt{\xi}} }}{\Gamma(x)^{\frac{1}{\xi}}}
     \Big(\int_0^T\!\!\frac{1}{(1-\e^{-K_0s})^{ d/\xi }}\d s\Big)\sup_{s\in[0,T]}\mu_0(G^{\xi}(s,\cdot))^{\frac 1\xi}<\infty.
   \end{align*}
   The proof is complete.
\end{proof}

\begin{mylem}\label{t-3.3}
Assume that there exists $\eta>2\lambda T d$ such that
\begin{equation}\label{3.17}
\sup_{t\in[0,T]}\mu_0\Big(\e^{\eta|Z(t,\cdot)|^2}\Big)<\infty.
\end{equation}
Then for every $\gamma>1$,
\begin{equation}\label{3.18}
\E\big[\sup_{0\leq t\leq T}|X(t)|^{\gamma}\big]<\infty.
\end{equation}
\end{mylem}

\begin{proof}
Denote by
\begin{equation}\label{3.18.5}
M_t=\int_0^t\la \sigma^{-1}Z(s,Y(s)),\d W(s)\raa,\quad \text{and}\quad \la M\raa_t=\int_0^t|\sigma^{-1}Z(s,Y(s))|^2\d s.
\end{equation}
Since \eqref{3.17} holds for $\eta>2\lambda Td$ and $|\sigma^{-1} Z(t,\cdot)|^2\leq \lambda |Z(t,\cdot)|^2$, by virtue of Lemma \ref{t-3.1}, we get
\[\E\e^{\frac 12\la M\raa_t}<\infty, \quad t\in(0,T],\]
which implies that $t\mapsto \exp\big(M_t-\frac 12 \la M\raa_t\big)$ is an exponential martingale for $t\in[0,T]$. The Girsanov theorem yields that $(\widehat W_1(t))_{t\leq T}$ is a Brownian motion under $\q_1$ and further $(Y(t))_{t\leq T}$ under $\q_1$ defined by \eqref{q1} is also a solution of \eqref{2.1}. Then the weak uniqueness of SDE \eqref{2.1} means that $(X(t))_{t\leq T}$ under $\p$ has the same distribution as $(Y(t))_{t\leq T}$ under $\q_1$. Denote by $\E=\E_\p$ the expectation w.r.t. $\p$, and by $\E_{\q_1}$ the expectation w.r.t. $\q_1$.
Therefore,
\begin{equation}\label{3.19}
\begin{aligned}
  \E_{ }\big[\sup_{0\leq t\leq T}|X(t)|^\gamma\big]&=\E_{\q_1}\big[\sup_{0\leq t\leq T} |Y(t)|^\gamma\big]=\E_{ }\Big[\frac{\d \q_1}{\d \p}\cdot\sup_{0\leq t\leq T}|Y(t)|^\gamma\Big]\\
  &\leq \E_{ }\Big[\Big(\frac{\d \q_1}{\d \p}\Big)^p\Big]^{\frac 1p}\,\E_{ }\Big[\sup_{0\leq t\leq T} |Y(t)|^{\gamma q}\Big]^{\frac 1q},\quad p,\,q>1,\ \frac 1p+\frac 1q=1.
\end{aligned}
\end{equation}
By \cite[Theorem 2.4.4]{Mao}, the global Lipschitz condition of $Z_0$ yields that
\begin{equation}\label{3.20}
\E_{}\big[\sup_{0\leq t\leq T} |Y(t)|^{\gamma q}\big]<\infty.
\end{equation}
Meanwhile,
\begin{equation}\label{3.21}
  \begin{split}
    \E\Big[\Big(\frac{\d \q_1}{\d \p}\Big)^{p}\Big]&=\E\big[\exp\big(pM_T-\frac{p}{2}\la M\raa_T\big)\big]\\
    &\leq \E\Big[\exp\Big(2p  M_T- 2p^2 \la M\raa_T\Big)\Big]^{\frac 1{2}}\E\Big[\exp\Big( p(2p-1) \la M\raa_T\Big)\Big]^{\frac{1}{2}}.
  \end{split}
  \end{equation}
As $\lim_{p\downarrow 1} 2p^2=2$ and $\lim_{p\downarrow 1} p(2p-1)=1$, there exists $p_0>1$ such that $2p_0^2 \lambda Td<\eta$ and hence $p_0(2p_0-1)\lambda Td<\eta$. So, by Lemma \ref{t-3.1},
\[\E\big[\exp(2p_0^2\la M\raa_T)\big]<\infty,\quad \E\big[\exp\big(p_0(2p_0-1)\la M\raa_T\big)\big]<\infty.\] This implies that $t\mapsto \exp\big(2p_0M_t-2p_0^2\la M\raa_t\big)$ is an exponential martingale for $t\in[0,T]$. Consequently, we derive from \eqref{3.21} that
\begin{equation}\label{3.22}
\E\Big[\Big(\frac{\d \q_1}{\d \p}\Big)^{p_0}\Big]<\infty.
\end{equation}
Inserting \eqref{3.22}, \eqref{3.20} into \eqref{3.19} by taking $p=p_0$ and $q=p_0/(p_0-1)$, we obtain the desired result \eqref{3.18}.
\end{proof}

\noindent\textbf{Proof of Theorem \ref{t-0}} Similar to the representation of $(X(t))$ through the auxiliary process $(Y(t))$ determined by \eqref{2.6}, $(\wt X(t))$ can also be represented through $(Y(t))$. Indeed, setting \[\wt W(t)=W(t)-\int_0^t\sigma^{-1}   \wt Z(s, Y(s))\d s  \quad \text{ for $t\in [0,T]$},\]
rewrite $(Y(t))$ as
\[\d Y(t)=\tilde b(t,Y(t))\d t+\sigma \d \wt W(t),\]
then
\begin{equation}\label{0-1}
\wt \q:=\exp\Big[\int_0^T\la \sigma^{-1}\wt Z(s,Y(s)),\d W(s)\raa-\frac 12\int_0^T|\sigma^{-1}\wt Z(s,Y(s))|^2\d s\Big]\p
\end{equation} is a probability measure if
\begin{equation}\label{0-2}
\E\Big\{\exp\Big[\frac 12\int_0^T|\sigma^{-1}\wt Z(s,Y(s))|^2\d s\Big]\Big\}<\infty.
\end{equation}
Moreover, by Lemma \ref{t-3.1}, under condition (H1), the estimate \eqref{0-2} holds. Therefore, $\wt \q$ is a probability measure and further $(\wt W(t))_{t\in[0,T]}$ is a new Brownian motion under $\wt Q$ according to the Girsanov theorem. The uniqueness of the solution to SDE \eqref{u-1} yields that $(Y(t))_{t\in[0,T]}$ under $\wt \q$ has the same distribution of $(\wt X(t))_{t\in[0,T]} $ under $\p$.

Consequently, for any bounded measurable function $f$ on $\R^d$ with $\|f\|_\infty\leq 1$,
it holds,
\begin{align*}
  \big|\E f(X(t))-\E f(\wt X(t))\big|&=\big|\E_{\q_1} f(Y(t))-\E_{\wt \q} f(Y(t))\big|\\
  &=\Big|\E\Big[\Big(\frac{\d \q_1}{\d \p}-\frac{\d \wt \q}{\d \p}\Big) f(Y(t))\Big]\Big|\\
  &\leq \E\Big|\frac{\d \q_1}{\d \p}-\frac{\d \wt\q}{\d \p}\Big|,\quad t\in[0,T].
\end{align*}  
Setting
\[M_T=\int_0^T\la\sigma^{-1}Z(s,Y(s)),\d W(s)\raa,\quad \wt M_T=\int_0^T\la \sigma^{-1}\wt Z(s,Y(s)),\d W(s)\raa,\]
and
\[\la M\raa_T=\int_0^T|\sigma^{-1}Z(s,Y(s))|^2 \d s,\ \ \la \wt M\raa_T=\int_0^T|\sigma^{-1}\wt Z(s,Y(s))|^2\d s,\]
by using $|\e^x-\e^y|\leq (\e^x+\e^y)|x-y|$ for all $x,\,y\in \R$, we obtain that
\begin{equation}\label{0-3}
\begin{split}
  &\big|\E f(X(t))-\E f(\wt X(t))\big|\\
  &\leq \E\Big[\Big(\frac{\d \q_1}{\d \p}+\frac{\d \wt\q}{\d \p}\Big)\big|M_T-\wt M_T-\frac 12\la M\raa_T+\frac12 \la \wt M\raa_T\big|\Big]\\
  &\leq \E\Big[\Big(\frac{\d \q_1}{\d \p}+\frac{\d \wt \q}{\d \p}\Big)^p\Big]^{\frac 1p}\E\big[\big|M_T-\wt M_T-\frac 12\la M\raa_T+\frac 12 \la \wt M\raa_T\big|^q\big]^{\frac 1q}
\end{split}
\end{equation} for $p,\,q>1$ with $1/p+1/q=1$.

For the first term in \eqref{0-3}, since $\eta>2\lambda Td$, the estimate \eqref{3.22} in Lemma \ref{t-3.3} implies that there exists some $p$ satisfying
\[1<p<p_0=\Big(\sqrt{\frac{\eta}{2\lambda T d}}\Big)\wedge 2\] such that
\begin{equation}\label{0-4}
\E\Big[\Big(\frac{\d \q_1}{\d \p}\Big)^p\Big]<\infty,\quad \text{and }\ \ \E\Big[\Big(\frac{\d \wt \q}{\d \p}\Big)^p\Big]<\infty.
\end{equation}

For the second term of \eqref{0-3}, let us consider first the estimate of $\E\big[|M_T-\wt M_T|^q\big]$ then the estimate of $\E\big[\big|\la M\raa_T-\la \wt M\raa_T\big|^q\big]$.
As $q=p/(p-1)>2$, it follows from Burkholder-Davis-Gundy's inequality and Jensen's inequality that
\begin{align*}
  \E\big[|M_T-\wt M_T|^q\big]
  &\leq C_q\E\Big[\Big(\int_0^T|\sigma^{-1}(Z-\wt Z)(s,Y(s))|^2\d s\Big)^{\frac{q}{2}}\Big]\\
  &\leq C_qT^{\frac q 2-1}\lambda^{\frac q2}\E\Big[\int_0^T|Z-\wt Z|^{q}(s,Y(s))\d s\Big].
\end{align*}
Because
\[\sup_{t\in[0,T]}\mu_0(\e^{\eta|Z(t,\cdot)|^2})<\infty,\quad \sup_{t\in[0,T]}\mu_0(\e^{\eta|\wt Z(t,\cdot)|^2})<\infty,\]
it follows that for every $\xi>d$
\begin{equation}\label{0-5}
\sup_{t\in [0,T]}\mu_0\big(|Z(t,\cdot)-\wt Z(t,\cdot)|^{\frac{q\xi}{2}}\big)<\infty.
\end{equation}
Then, the inequality \eqref{n-1} in Lemma \ref{t-3.1} implies that
\begin{equation*}
\E\Big[ |Z(s,Y(s))-\wt Z(s,Y(s))|^{q} \Big]\leq \frac{\e^{\frac{K_0}{\xi-\sqrt{\xi}} }\mu_0\big(|Z(s,\cdot)-\wt Z(s,\cdot)|^{ q\xi }\big)^{\frac 1\xi}}{\Gamma(x_0)^{\frac 1\xi}(1-\e^{-K_0s})^{\frac d\xi}},\quad s\in[0,T],
\end{equation*} where $x_0$ denotes the initial value of $(Y(t))$.
Therefore,
\begin{equation}\label{0-6}
\E\big[|M_T-\wt M_T|^q\big]\leq C_q T^{\frac q2-1}\lambda^{\frac q2}\int_0^T\frac{\d s}{(1-\e^{-K_0s})^{\frac{d}{\xi}}}\Big(\int_{\R^d}|Z(s,y)-\wt Z(s,y)|^{ q\xi }\d \mu_0(y)\Big)^{\frac 1\xi}.
\end{equation}

To proceed,
\begin{align*}
  &\E\big[\big|\la M\raa_T-\la \wt M\raa_T\big|^q\big]\\
  &\leq \E\Big[\Big(\int_0^T|\sigma^{-1}(Z-\wt Z)(s,Y(s))|\big(|\sigma^{-1}Z|+|\sigma^{-1}\wt Z|\big)(s,Y(s))\d s\Big)^q\Big]\\
  &\leq \E\Big[\Big(\!\int_0^T\!\!|\sigma^{-1}(Z\!-\!\wt Z)|^{\gamma}(s,Y(s))\d s\Big)^{q}\Big]^{\frac1\gamma}\E\Big[\Big(\!\int_0^T \!\! \big(|\sigma^{-1}Z|\!+\!|\sigma^{-1}\wt Z|\big)^{\gamma'}(s,Y(s))\d s\Big)^{q}\Big]^{\frac{1}{\gamma'}},
\end{align*} where $\gamma,\,\gamma'>1$ satisfy $1/\gamma+1/\gamma'=1$.
By \eqref{0-5} and Lemma \ref{t-3.1},
\begin{equation}\label{0-7}
\E\Big[\Big(\int_0^T\big(|\sigma^{-1}Z|+|\sigma^{-1}\wt Z|\big)^{\gamma'}(s,Y(s))\d s\Big)^q\Big]^{\frac{1}{\gamma'}}<\infty.
\end{equation} Applying \eqref{n-1}  and Jensen's inequality again, we get, for every $\xi>d$,
\begin{equation}\label{0-8}
\begin{split}
  &\E\Big[\Big(\int_0^T|\sigma^{-1}(Z-\wt Z)|^\gamma(s,Y(s))\d s\Big)^{q}\Big]^{\frac1{\gamma}}\\
   &\leq \lambda^{\frac q2}T^{\frac{q-1}{\gamma}}\frac{\e^{\frac{K_0}{\xi-\sqrt{\xi}} }}{\Gamma(x_0)^{\frac 1\xi}} \Big\{\int_0^T\frac{\d s}{(1-\e^{-K_0 s})^{\frac d\xi}}\Big(\int_{\R^d}|Z(s,y)-\wt Z(s,y)|^{\gamma q\xi}\d\mu_0(y)\Big)^{\frac 1\xi}\Big\}^{\frac1\gamma}.
\end{split}
\end{equation}
Invoking \eqref{0-6}, \eqref{0-7} and \eqref{0-8},  for every $\xi>d$, $\gamma>1$, there exists a constant $C=C(K_0,T, \lambda,\xi,\gamma,q)$ such that
\begin{equation}\label{0-9}
\begin{split}
  &\E\big[|M_T-\wt M_t-\frac 12\la M\raa_T+\frac 12\la \wt M\raa_T|^q\big]\\
  &\leq C\big(\E[|M_T-\wt M_T|^q]+\frac 12\E[|\la M\raa_T-\la \wt M\raa_T|^q]\big)\\
  &\leq C\Big\{ \int_0^T\frac{\mu_0\big(|Z-\wt Z|^{q\xi}(s,\cdot)\big)^{\frac 1\xi}}{(1-\e^{-K_0s})^{\frac{d}\xi}}\,\d s+\Big(\int_0^T\frac{\mu_0\big(|Z-\wt Z|^{\gamma q\xi}(s,\cdot)\big)^{\frac 1\xi}}{(1-\e^{-K_0s})^{\frac d\xi}}\,\d s\Big)^{\frac 1\gamma}\Big\}
\end{split}
\end{equation}
Consequently, inserting \eqref{0-4}, \eqref{0-9} into \eqref{0-3}, we arrive at
\begin{equation}\label{0-10}
\begin{split}
&|\E f(X(t))-\E f(\wt X(t))|\\
&\leq C\Big\{ \int_0^T\frac{\mu_0\big(|Z-\wt Z|^{q\xi}(s,\cdot)\big)^{\frac 1\xi}}{(1-\e^{-K_0s})^{\frac{d}\xi}}\,\d s+\Big(\int_0^T\frac{\mu_0\big(|Z-\wt Z|^{\gamma q\xi}(s,\cdot)\big)^{\frac 1\xi}}{(1-\e^{-K_0s})^{\frac d\xi}}\,\d s\Big)^{\frac 1\gamma}\Big\}^{\frac 1{ q}}.
\end{split}
\end{equation}
Letting $\gamma \ra 1$ and $q\ra q_0=p_0/(p_0-1)$, we have
\begin{equation}\label{0-11}
|\E f(X(t))-\E f(\wt X(t))|\leq C\Big\{\int_0^T\!\frac{\mu_0(|Z-\wt Z|^{q_0\xi}(s,\cdot))^{\frac 1\xi}}{(1-\e^{-K_0s})^{\frac{d}{\xi}}}\d s\Big\}^{\frac 1{q_0}}.
\end{equation}
Taking the supremum of $f$ in \eqref{0-10} over all bounded measurable functions $f$ with $\|f\|_{\mathrm{Lip}}+\|f\|_\infty\leq 1$,
  the formula \eqref{bL-dis}  yields  \eqref{u-2} and hence completes the proof. \fin

In order to prove Theorem \ref{t-1} and Theorem \ref{t-2}, we have to make some prepartions concerning the representation \eqref{rep-EM}-\eqref{q2} of EM's approximation $(X_\de(t))$.
\begin{mylem}\label{t-3.2}Suppose $(H_\sigma)$ holds. Let $V\in\mathscr V$, $T>0$, $\beta>0$ be given.
Assume there exist $\eta_0>0$ and $\eta>2\beta \lambda T d$ such that
\[\mu_0(\e^{\eta_0|Z_0|^2})<\infty\quad \text{and}\quad \sup_{t\in[0,T]}\mu_0\big(\e^{\eta |Z(t,\cdot)|^2}\big)<\infty.\] Then
\begin{equation}\label{3.8}
\E\Big[\exp\Big(\beta\int_0^T |\sigma^{-1}(Z_0(Y(s))-b(s_\de,Y(s_\de)))|^2\d s\Big)\Big]<\infty
\end{equation}
if
\[\de<\frac{\eta_0}{8K_0\beta\lambda^2\gamma_0 T\eta_0+2\beta\lambda \gamma_0 K_0(T+1)d}\wedge 1,\]
where $\gamma_0=\frac{\eta}{\eta-2\beta\lambda T d}$. In particular, when $Z_0$ is bounded, then  \eqref{3.8} holds for all $\de\in (0,1)$.
\end{mylem}

\begin{proof}
  By H\"older's inequality, it holds
  \begin{equation}\label{3.9}
\begin{split}
  &\E\exp\Big[\beta\int_0^T|\sigma^{-1}(Z_0(Y(s))-b(s_\de,Y(s_\de)))|^2\d s\Big]\\
  &=\E\exp\Big[\beta\int_0^T |\sigma^{-1}(Z_0(Y(s))-Z_0(Y(s_\de))-Z(s_\de,Y(s_\de)))|^2\Big]\\
  &\leq \E\exp\Big[2\beta\lambda\int_0^T\big(|Z_0(Y(s))-Z_0(Y(s_\de))|^2+|Z(s_\de,Y(s_\de))|^2\big)\d s\Big]\\
  &\leq \Big\{\E\big[\e^{2\beta\lambda\gamma \int_0^T|Z_0(Y(s))-Z_0(Y(s_\de))|^2\d s}\big]\Big\}^{\frac 1\gamma}\Big\{\E\big[\e^{2\beta\lambda\gamma'\int_0^T|Z(s_\de,Y(s_\de))|^2\d s}\big]\Big\}^{\frac1{\gamma'}}\\
  &=:\mathbf{I}\cdot\mathbf{II},
  \end{split}
\end{equation} where $\gamma,\,\gamma'>1$ with $1/\gamma+1/\gamma'=1$.

First, let us consider the second term $\mathbf{II}$.
Under the condition that for some $\eta>\lambda T d$ such that
\[\sup_{t\in[0,T]}\mu_0\big(\e^{\eta |Z(t,\cdot)|^2}\big)<\infty,\]
Take  $\gamma'>1$ such that $2\beta\lambda \gamma' T d<\eta$.  According to \eqref{3.5} by replacing $G$ with $Z(t_\de,\cdot)$,
\[\E \e^{2\beta\lambda \gamma' |Z(t_\de,Y(t_\de))|^2}
\leq \Gamma(x)^{-\frac 1p}\frac{\e^{\frac{K}{p-\sqrt p} }}{(1-\e^{-K t_\de})^{d/p}}\sup_{t\in[0,T]}\mu_0\Big(\e^{2\beta\lambda\gamma'|Z(t_\de,\cdot)|^2}\Big)^{\frac 1p}<\infty, \quad t_\de >0.\]
Similar to \eqref{3.7} by replacing $\beta$ there with $2\beta\lambda\gamma'$, we   deduce that
\[\E\big[\e^{2\beta\lambda\gamma'\int_0^T|Z(s_\de,Y(s_\de))|^2\d s}\big]<\infty.\]
Therefore, the second term $\mathbf{II}<\infty$ if $1<\gamma'<\eta/(2\beta\lambda T d)$.

Next, we go to study the term $\mathbf{I}$.
For $s\in(0,T]$, \eqref{2.5} yields
\[Y(s)=Y(s_\de)+\int_{s_\de}^s Z_0(Y(r))\d r +\sigma (W(s)-W(s_\de)),
\]
then
\[|Y(s)-Y(s_\de)|^2\leq 2 \Big(\int_{s_\de}^s |Z_0(Y(r))|\d r\Big)^2+ 2\lambda|W(s)-W(s_\de)|^2.\]
As $V\in\mathscr V$,
\begin{equation}\label{3.10}
\begin{split}
  \mathbf{I}\,{}^\gamma&=\E\Big[\e^{2\beta\lambda \gamma\int_0^T |Z_0(Y(s))-Z_0(Y(s_\de))|^2\d s}\Big]
   \leq \E\Big[\e^{2\beta\lambda \gamma K_0\int_0^T|Y(s)-Y(s_\de)|^2\d s}\Big]\\
  &\leq \E\Big[\exp\Big(4\beta\lambda \gamma K_0\int_0^T\big(\int_{s_\de}^s|Z_0(Y(r))|^2\d r+\lambda |W(s)-W(s_\de)|^2 \big)\d s\Big)\Big]\\
  &\leq \E\Big[\e^{4p\beta\lambda \gamma K_0\int_0^T\big(\int_{s_\de}^s|Z_0(Y(r))|\d r\big)^2 \d s}\Big]^{\frac 1p}\cdot\E\Big[\e^{4q\beta\lambda^2 \gamma K_0\int_0^T|W(s)-W(s_\de)|^2\d s}\Big]^{\frac 1q},
\end{split}
\end{equation} where $p,\,q>1$ with $1/p+1/q=1$.

On one hand, direct calculation yields that
\begin{equation}\label{3.11}
\begin{aligned}
\E\Big[\e^{4p\beta\lambda \gamma K_0\int_0^T\big(\int_{s_\de}^s|Z_0(Y(r))|\d r\big)^2 \d s}\Big]
  &\leq \E\Big[\e^{4p\beta\lambda\gamma K_0\int_0^T(s-s_\de)\int_{s_\de}^s|Z_0(Y(r))|^2\d r \d s}\Big]\\
  &\leq \E\Big[\e^{4p\beta\lambda \gamma K_0\de(T+\de)\int_0^T|Z_0(Y(r))|^2 \d r}\Big].
\end{aligned}
\end{equation}
According to Lemma \ref{t-3.1}, if
\begin{equation}\label{3.12}
4p\beta\lambda \gamma K_0\de (T+\de) T d<  \eta_0,
\end{equation}
then it follows from $\mu_0(\e^{\eta_0 |Z_0|^2})<\infty$ that
\[\E\Big[\e^{4p\beta\lambda \gamma K_0\int_0^T\big(\int_{s_\de}^s|Z_0(Y(r))|\d r\big)^2 \d s}\Big]\leq \E\Big[\e^{4p\beta\lambda \gamma K_0\de(T+\de)\int_0^T|Z_0(Y(r))|^2 \d r}\Big]<\infty.\]

On the other hand, letting $\theta=4q\beta\lambda^2\gamma K_0$,
if
\begin{equation}\label{3.13}
2\de \theta T<1,
\end{equation}
then
\begin{equation}\label{3.14}
\begin{split}
  \E\Big[\e^{\theta\int_0^T|W(s)-W(s_\de)|^2\d s}\Big]&\leq \frac 1 T\int_0^T \E\Big[\e^{\theta T|W(s)-W(s_\de)|^2\d s}\Big]\d s\\
  &=\frac 1 T\int_0^T\int_{\R^d}\frac 1{(2\pi)^{d/2}}\e^{-\frac{(1-2(s-s_\de)\theta T)x^2}{2}}\d x \d s<\infty.
\end{split}
\end{equation}

In conclusion, in order to ensure $\mathbf{II}<\infty$,    $\gamma'$ should satisfy $1<\gamma'<\frac{\eta}{2\beta\lambda T d}$, which means further that $\gamma$  satisfies $\gamma>\gamma_0:=\frac{\eta}{\eta-2\beta\lambda T d}$. For the purpose that $\mathbf{I}<\infty$, $p,\,q$ must be chosen so that   \eqref{3.12} and \eqref{3.13} hold, that is,
\begin{equation}\label{3.15}
\de<\frac{\eta_0}{4p\beta\lambda \gamma_0 K_0(T+1)}\wedge 1,\quad \text{and}\quad \de<\frac{1}{8 q K_0 \beta\lambda^2\gamma_0 T}.
\end{equation}
Thus, $\de$ must satisfy
\[\de<\sup\Big\{ \frac{\eta_0}{4p\beta\lambda \gamma_0K_0(T+1)}\bigwedge \frac{1}{8q K_0\beta\lambda^2\gamma_0 T};\ p,q>1, \frac 1 p+\frac 1q=1\Big\}\wedge 1,\]
where $a\wedge b=\min\{a,b\}$. Taking the optimal choice of $p,q$ in the previous inequality, we get
\begin{equation}\label{3.16}\de< \frac{\eta_0}{8K_0\beta\lambda^2\gamma_0 T\eta_0+2\beta\lambda \gamma_0 K_0(T+1)d}\wedge 1.
\end{equation}
This means that when $\de$ satisfies \eqref{3.16}, there exists $p,\,q>1$ with $1/p+1/q=1$ such that $\mathbf{I}<\infty$. In all, when $\eta>2\beta\lambda T d$ and $\de$ satisfies \eqref{3.16}, there exist $\gamma>\gamma_0>1$ and $\gamma'=\gamma/(\gamma-1)$ such that $\mathbf{II}<\infty$ and $\mathbf{I}<\infty$. Then the desired estimate \eqref{3.8} follows from \eqref{3.9} immediately.

At last, when $Z_0$ is bounded, it holds  $\mu_0(\e^{\eta_0|Z_0|^2})<\infty$ for all $\eta_0>0$. Hence, for every $\de\in (0,1)$, \eqref{3.16} holds by taking $\eta_0$ large enough. Thus,  \eqref{3.8} holds.
\end{proof}

\begin{mylem}\label{t-3.4}
Assume that there exist $\eta_0>0$ and $\eta>4\lambda Td$ such that
\begin{equation}\label{3.4.1}
\mu_0\big(\e^{\eta_0|Z_0|^2}\big)<\infty\quad \text{and}\quad \sup_{0\leq t\leq T}\mu_0\big(\e^{\eta|Z(t,\cdot)|^2}\big)<\infty.
\end{equation}
Suppose
\[\de <\frac{\eta_0}{16K_0\lambda^2\gamma_0T\eta_0+4\lambda\gamma_0K_0(T+1)d}\wedge 1,
\]
where $\dis \gamma_0=\frac{\eta}{\eta-4\lambda Td}$.
Then, for every $\gamma>1$, EM's approximation $(X_\de(t))$ satisfies
\begin{equation}\label{3.4.2}
\E\big[\sup_{0\leq t\leq T}|X_\de(t)|^\gamma\big]<\infty.
\end{equation}
\end{mylem}

\begin{proof}
  Let
  \begin{equation}\label{3.4.2.5}
  \begin{split}
  \widehat M_t&=-\int_0^t\la \sigma^{-1}\big(Z_0(Y(s))-b(s_\de,Y(s_\de))\big),\d W(t)\raa,\\
  \la\widehat M\raa_t&=\int_0^t\big|\sigma^{-1}\big(Z_0(Y(s))-b(s_\de,Y(s_\de))\big)\big|^2\d s, \quad t\in[0,T].
  \end{split}
  \end{equation}
  Under the hypothesis of this lemma, by virtue of Lemma \ref{t-3.2}, it holds
  \[\E\e^{\frac 12\la \,\widehat M\,\raa_t}<\infty,\quad t\in (0,T],\]
  which yields that $t\mapsto \exp\big(\widehat  M_t-\frac 12 \la \widehat  M\raa_t\big)$ is an exponential martingale for $t\in [0,T]$. Using the Girsanov theorem and  the representation \eqref{rep-EM} of $(X_\de(t))$ in terms of $(Y(t))$, we have that $(X_\de(t))_{t\leq T}$ under $\p$ admits the same distribution as $(Y(t))_{t\leq T}$ under $\q_2$.
  Consequently,
  \begin{equation}\label{3.4.3}
  \begin{split}
    \E\big[\sup_{0\leq t\leq T}|X_\de(t)|^\gamma\big]&=\E_{\q_2}\big[\sup_{0\leq t\leq T}|Y(t)|^\gamma\big]=\E\Big[\frac{\d \q_2}{\d \p}\cdot\sup_{0\leq t\leq T}|Y(t)|^{\gamma}\Big]\\
    &\leq \E\Big[\Big(\frac{\d \q_2}{\d \p}\Big)^p\Big]^{\frac 1p}\,\E\Big[\sup_{0\leq t\leq T}|Y(t)|^{\gamma q}\Big]^{\frac 1q},\quad p,q>1,\ \frac 1p+\frac 1q=1.
  \end{split}
  \end{equation}
 Then, following the same line of the proof of Lemma \ref{t-3.3}, whereas in this situation applying Lemma \ref{t-3.2} instead of Lemma \ref{t-3.1}, we can prove \eqref{3.4.2}. The details are omitted.
\end{proof}

\noindent\textbf{Proof of Theorem \ref{t-1}}\ Under the hypothesis of this theorem, $(X(t))_{t\leq T}$ under $\p$ has the same distribution as $(Y(t))$ under $\q_1$, and $(X_\de(t))_{t\leq T}$ under $\p$ has the same distribution as $(Y(t))$ under $\q_2$.
Therefore, for every bounded measurable function $f$ on $\R^d$,
\begin{equation}\label{M-1}
\begin{split}
  |\E f(X(t))-\E f(X_\de(t))|&=|\E_{\q_1} f(Y(t))-\E_{\q_2}f(Y(t))|\\
  &\leq K_f\E\Big[\Big|\frac{\d\q_1}{\d \p}-\frac{\d \q_2}{\d \p}\Big|\Big],
\end{split}
\end{equation}
where $K_f=\sup\{|f(x)|;\,x\in \R^d\}<\infty$. Recall that $M_T$, $\widehat M_T$ be defined  by \eqref{3.18.5} and \eqref{3.4.2.5} respectively.
By the inequality $|\e^x-\e^y|\leq (\e^x+\e^y)|x-y|$ for all $x,\,y\in\R$, we get
\begin{equation}\label{M-2}
\begin{split}
  \E\Big[\Big|\frac{\d \q_1}{\d \p}-\frac{\d \q_2}{\d \p}\Big|\Big]&\leq
  \E\Big[\Big(\frac{\d \q_1}{\d \p}+\frac{\d \q_2}{\d \p}\Big)\big|M_T-\tilde M_T-\frac 12 \la M\raa_T+\frac 12\la \tilde M\raa_T\big|\Big]\\
  &\leq \E\Big[\Big(\frac{\d \q_1}{\d \p}+\frac{\d \q_2}{\d \p}\Big)^p\Big]^{\frac 1p}\E\big[ \big|M_T-\tilde M_T-\frac 12 \la M\raa_T+\frac 12\la \tilde M\raa_T\big|^q\big]^{\frac 1q}
\end{split}
\end{equation} for $p,\,q>1$ with $1/p+1/q=1$.

According to the hypothesis $\eta>4\lambda Td$, there exists $p_0>1$ such that
$ 4p_0^2\lambda T d<\eta$.
Invoking the estimate \eqref{3.22}, we know that
\[\E\Big[\Big(\frac{\d \q_1}{\d \p}\Big)^{p_0}\Big]<\infty.\]
Furthermore, by H\"older's inequality,
\[\E\Big[\Big(\frac{\d \q_2}{\d \p}\Big)^{p_0}\Big] \leq \E\big[\exp\big(2p_0\tilde M_T-2p_0^2\la \tilde M\raa_T\big)\big]^{\frac 12}\E\big[\exp\big(p_0(2p_0-1)\la \tilde M\raa_T\big)\big]^{\frac 12}.\]
Due to Lemma \ref{t-3.2}, for $\de>0$ satisfying
\[\de<\frac{\eta_0}{16K_0p_0^2\lambda^2\gamma_0 T\eta_0+4p_0^2\lambda \gamma_0 K_0(T+1)d}\wedge 1,\]
it holds
\[\E\Big[\Big(\frac{\d \q_2}{\d \p}\Big)^{p_0}\Big]<\infty.\]

Next, we proceed to show that for any $q'>1$
\begin{equation}\label{M-3}
\E\big[|M_T-\tilde M_T-\frac 12\la M\raa_T+\frac 12\la \tilde M\raa_T|^{q'}\big]<\infty.
\end{equation}
Indeed, for any $\veps>0$, $q'>1$, there exists a constant $C=C(q',\veps)$ such that
$|y|^{q'}\leq C\e^{\veps |y|}$ for all $ y\in\R$.
Hence, by Burkholder-Davis-Gundy's inequality,
\begin{equation}\label{M-4}
\begin{split}
  &\E\big[|M_T-\tilde M_T-\frac 12\la M\raa_T+\frac 12\la \tilde M\raa_T|^{q'}\big]\\
  &\leq C\big(\E[|M_T|^{q'}]+\E[|\tilde M|^{q'}]+\frac 12\E[\la M\raa_T^{q'}]+\frac 12\E[\la \tilde M\raa_T^{q'}]\big)\\
  &\leq C\big(\E[\la M\raa_T^{q'/2}]+\E[\la \tilde M\raa_T^{q'/2}]+ \E[\la M\raa_T^{q'}]+ \E[\la \tilde M\raa_T^{q'}] \big)\\
  &\leq C\big(\E\e^{\veps \la M\raa_T} +\E\e^{\veps \la \tilde M\raa_T}\big).
\end{split}
\end{equation}
By virtue of Lemma \ref{t-3.1} and Lemma \ref{t-3.2}, for $0<\veps<\frac{\eta}{4\lambda T d}$ and
\[\de<\frac{\eta_0}{16K_0\veps \lambda^2\gamma_0T\eta_0+4\veps \lambda\gamma_0 K_0(T+1)d}\wedge 1,\]
we have
$\E\e^{\veps \la M\raa_T}<\infty$ and $\E\e^{\veps \la \tilde M\raa_T}<\infty$, which yields \eqref{M-3} by \eqref{M-4}.

Since the pathes of the process $(Y(t))$ are almost surely continuous, by the continuity of $x\mapsto Z_0(x)$ and $(t,x)\mapsto b(t,x)$,
\begin{align*}
  \lim_{\de\downarrow 0} \int_0^T\!\la \sigma^{-1}Z(s,Y(s)),\d W(s)\raa\!-\!\int_0^T\!\la\sigma^{-1}(Z_0(Y(s))\!-\!b(s_\de,Y(s_\de))),\d W(s)\raa =0,\ \ \text{a.s.}
\end{align*}
and
\begin{equation*}
\lim_{\de\downarrow 0} \int_0^T\Big(|\sigma^{-1} Z_0(Y(s))|^2-\big|\sigma^{-1}\big(Z_0(Y(s))-b(s_\de, Y(s_\de))\big)\big|^2\Big)\d s=0, \quad\text{a.s.}
\end{equation*}
Furthermore, \eqref{M-3} implies the uniform integrability of $ M_T-\tilde M_T-\frac 12\la M\raa_T+\frac 12\la \tilde M\raa_T$ w.r.t. $\de>0$. Applying \eqref{M-2} with $p=p_0$, we obtain
\begin{equation}\label{M-5}
  \lim_{\de\downarrow 0} \E\Big[\Big|\frac{\d \q_1}{\d \p}-\frac{\d \q_2}{\d \p}\Big|\Big]=0,
\end{equation}
and further,   by \eqref{M-1},
\begin{equation*}
  \lim_{\de\downarrow 0}|\E f(X(t))-\E f(X_\de(t))|=0.
\end{equation*}
Due to \eqref{M-1} and   \eqref{bL-dis}, \eqref{M-5} yields further
\[\lim_{\de\downarrow 0} W_{bL}(\mathscr L(X_\de(t)),\mathscr L(X(t)))=0,\quad t\in [0,T].\]
This concludes the proof. \fin

\noindent\textbf{Proof of Theorem \ref{t-2}}\ For bounded measurable function $f$ with $\|f\|_\infty\leq  1$, by \eqref{M-1}
\begin{equation}\label{j-1}
\begin{split}
  &|\E f(X(t))-\E f(X_\de(t))|\leq \E\Big[\Big|\frac{\d \q_1}{\d \p}-\frac{\d \q_2}{\d \p}\Big|\Big]\\
  &\leq  \E\Big[\Big(\frac{\d \q_1}{\d \p}+\frac{\d \q_2}{\d \p}\Big)^p\Big]^{\frac 1p}\E\big[|M_T-\tilde M_T-\frac 12\la M\raa_T+\frac 12 \la \tilde M\raa_T|^q\big]^{\frac 1q},
\end{split}
\end{equation}where $p,q>1$, $1/p+1/q=1$,  $M_T$, $\tilde M_T$ are defined in the proof of Theorem \ref{t-1}. As shown in the argument of Theorem \ref{t-1}, there exists $p_0>1$ satisfying
$2p_0^2\lambda Td<\eta_0$ such that
\[\E\Big[\Big(\frac{\d \q_1}{\d \p}\Big)^{p_0}\Big]<\infty,\quad \E\Big[\Big(\frac{\d\q_2}{\d \p}\Big)^{p_0}\Big]<\infty,
\] if $\de$ is sufficiently so that
\[\de<\frac{\eta_0}{16K_0p_0^2\lambda^2\gamma_0 T\eta_0+4p_0^2\lambda \gamma_0 K_0(T+1)d}\wedge 1, \quad \text{where}\ \gamma_0=\frac{\eta}{\eta-4p_0^2\lambda T d}. \]
Set $q_0=p_0/(p_0-1)$. Without loss of generality, $p_0$ can be taken sufficiently small so that $q_0>2$.  Therefore, the convergence rate is determined by the term
\begin{align*}
&\E\big[|M_T-\tilde M_T-\frac 12\la M\raa_T+\frac 12 \la \tilde M\raa_T|^{q_0}\big]^{\frac 1{q_0}}\\
&=\E\Big[\Big(\int_0^T\la \sigma^{-1}\big(Z(s,Y(s))-Z(s_\de,Y(s_\de))-Z_0(Y(s))+Z_0(Y(s_\de))\big),\d W(s)\raa\\
&\quad -\frac 12\! \int_0^T\!\!|\sigma^{-1}Z(s,Y(s))|^2\d s\!+\!\frac 12\int_0^T\!\!\!|\sigma^{-1}\big(Z_0(Y(s))\!-\!Z_0(Y(s_\de))\!-\!Z(s_\de,Y(s_\de))\big)|^2\d s\Big)^{q_0}\Big]^{\frac 1{q_0}}.
\end{align*}

In what follows, we use $C$ to denote a generic positive constant, whose value may be different from line to line.

First, we estimate the term $\E[|M_T-\tilde M_T|^{q_0}]$, and the term $\E[|\la M\raa_T-\la \tilde M\raa_T|^{q_0}]$ will be estimated in next step.  By $\mathrm{(H_\sigma)}$ and Burkholder-Davis-Gundy's inequality,
\begin{equation}\label{j-2}
\begin{split}
&\E\big[|M_T-\tilde M_T|^{q_0}\big]\\&\leq C\lambda^{\frac{q_0}{2}}
\E\Big[\Big(\int_0^T\!\!\big|Z(s,Y(s))-Z(s_\de,Y(s_\de))-Z_0(Y(s))+Z_0(Y(s_\de))\big|^2\d s\Big)^{\frac{q_0}{2}}\Big]\\
&\leq C\lambda^{\frac{q_0}{2}}\E\Big[\Big(\int_0^T\!\!|Z(s,Y(s))\!-\!Z(s_\de,Y(s_\de))|^2\d s\Big)^{\frac{q_0}{2}}\!\!+\!\Big( \!\int_0^T \!\!|Z_0(Y(s))\!-\!Z_0(Y(s_\de))|^2\Big)^{\frac{q_0}{2}}\Big].
\end{split}
\end{equation}
Put $N_T=[ T/\de]$, $t_k=k\de$ for $k=0,\ldots,N_T$, and $t_{N_T+1}=T$ for convenience of notation.
By Jensen's inequality,
\begin{equation}\label{j-3}
\begin{aligned}
  &\E\Big[\Big(\int_0^T|Z_0(Y(s))-Z_0(Y(s_\de))|^2\d s\Big)^{\frac{q_0}{2}}\Big]\\
  &\leq T^{\frac{q_0}{2}-1}\E\Big[\int_0^T|Z_0(Y(s))-Z_0(Y(s_\de))|^{q_0}\d s\Big]\\
  &\leq T^{\frac{q_0}{2}-1}\E\Big[\sum_{k=0}^{N_T}\int_{t_k}^{t_{k+1}} |Z_0(Y(s))-Z_0(Y(s_\de))|^{q_0}\d s\Big].
\end{aligned}
\end{equation}
Denote by $\rho(s,x,y)$ the density of $\p(Y(s)\in\d y|Y(0)=x)$. According to the classical theory of heat kernel estimate (cf. e.g. \cite{Ar,CHXZ,ZhQ}), when $\de$ is small enough, for $s\in (k\de,(k+1)\de)$, there exist constants $c_1,c_2>0$ such that
\begin{equation}\label{j-4}
c_1(2\pi(s-k\de))^{-\frac{d}{2}}\e^{\frac{|y|^2}{2(s-k\de)}}\leq \rho(s-k\de,x,x+y)\leq c_2(2\pi(s-k\de))^{-\frac{d}{2}}\e^{\frac{|y|^2}{2(s-k\de)}},\quad x,y\in\R^d.
\end{equation}
This yields that for every $\gamma>1$, $s\in (k\de, (k+1)\de]$,
\begin{equation}\label{j-5}
\begin{split}
\int_{\R^d}\!|y|^\gamma\rho(s-k\de,x,x+y)\d y&\leq C\int_{\R^d}\! \frac{|y|^\gamma }{(2\pi(s-k\de))^{d/2}}\e^{-\frac{|y|^2}{2(s-k\de)}}\d y
\leq C (s-k\de)^{\frac{\gamma}{2}}.
\end{split}
\end{equation}
Due to  \eqref{j-5}, it follows from \eqref{j-3} that
\begin{equation}\label{j-6}
\begin{split}
  &\E\Big[\Big(\int_0^T|Z_0(Y(s))-Z_0(Y(s_\de))|^2\d s\Big)^{\frac{q_0}{2}}\Big]\\
  &\leq T^{\frac{q_0}{2}-1}K_0\sum_{k=0}^{N_T}\int_{t_k}^{t_{k+1}}\E\big[|Y(s)-Y(s_\de)|^{q_0}\big]\d s\\
  &= T^{\frac{q_0}{2}-1}K_0\sum_{k=0}^{N_T}\int_{t_k}^{t_{k+1}}\!\!\!\int_{\R^d}\!\int_{\R^d} |y|^{q_0} \rho(k\de,x_0,x)\rho(s\!-\!k\de,x,x\!+\!y)\d x\d y \d s\\
  &\leq C T^{\frac{q_0}{2}-1}K_0\sum_{k=0}^{N_T}\int_{t_k}^{t_{k+1}}(s-t_k)^{\frac{q_0}{2}}\d s\\
  &\leq C(T, K_0, q_0) \de^{\frac{q_0}{2}},
\end{split}
\end{equation} where we denote $Y(0)=x_0$ and use the homogeneity of the process $(Y(t))$.

We proceed to prove
\begin{equation}\label{j-7}
  \E\Big[\Big(\int_0^T\!\!|Z(s,Y(s))\!-\!Z(s_\de,Y(s_\de))|^2\d s\Big)^{\frac{q_0}{2}}\Big]\leq C \de^{\frac{q_0}{2}\wedge (\alpha q_0)}.
\end{equation}
In fact, by Jensen's inequality,
\begin{equation}\label{j-8}
\begin{split}
  &\E\Big[\Big(\int_0^T\!\!|Z(s,Y(s))\!-\!Z(s_\de,Y(s_\de))|^2\d s\Big)^{\frac{q_0}{2}}\Big]\\
  &\leq T^{\frac{q_0}{2}-1}\sum_{k=0}^{N_T}\!\!\int_{t_k}^{t_{k+1}}
  \E\big[|Z(s,Y(s))-Z(s_\de,Y(s_\de))|^{q_0}\big]\d s\\
  &\leq C T^{\frac{q_0}{2}-1}\sum_{k=0}^{N_T}\int_{t_k}^{t_{k+1}}\!\!\E\big[|Z(s,Y(s))\!-\!Z(s_\de,Y(s))|^{q_0}
  \!+\!|Z(s_\de,Y(s))\!-\!Z(s_\de,Y(s_\de))|^{q_0}\big]\d s.
\end{split}
\end{equation}
By virtue of \eqref{2.8} and  \eqref{j-5},
\begin{equation}\label{j-9}
\begin{split}
  &\E\big[|Z(k\de,Y(s))-Z(k\de,Y(k\de))|^{q_0}\big]\\
  &\leq CK_1^{q_0}\!\int_{\R^d}\!\int_{\R^d}(1\!+\!|x|^{m_1}\!+\!|x+y|^{m_1})^{q_0}|y|^{q_0}
  \rho(k\de,x_0,x)\rho(s-k\de,x,x+y)\d x\d y\\
  &\leq CK_1^{q_0}\!\int_{\R^d}\!\int_{\R^d}\!\!\big(|y|^{q_0}\!+\!|x|^{m_1q_0}|y|^{q_0}\!+\!|y|^{(m_1+1)q_0}\big)
  \rho(k\de,x_0,x)\rho(s\!-\!k\de,x,x\!+\!y)\d x\d y\\
  &\leq CK_1^{q_0}(s-k\de)^{\frac{q_0}2}\int_{\R^d}\big(1+|x|^{m_1q_0}\big)\rho(k\de,x_0,x)\d x\\
  &= C K_1^{q_0}(s-k\de)^{\frac{q_0}{2}}\big(1+\E\big[|Y(k\de)|^{m_1q_0}\big]\big)\\
  &\leq CK_1^{q_0}(s-k\de)^{\frac{q_0}{2}},
\end{split}
\end{equation} where in the last step we have used the fact
$\E\big[\sup_{0\leq t\leq T} |Y(t)|^{\gamma}\big]<\infty$
for every $\gamma>1$ (cf. \cite[Theorem 2.4.4]{Mao}). Moreover, applying \eqref{2.8} again, we get
\begin{equation}\label{j-10}
\begin{split}
&\E\big[|Z(s,Y(s))-Z(k\de,Y(s))|^{q_0}\big] \leq (s-k\de)^{\alpha q_0}\E\big[h(Y(s))^{q_0}\big]\leq C(s-k\de)^{\alpha q_0},
\end{split}
\end{equation} where in the last step we used the fact $h$ is polynomial bounded and $\E\big[\sup_{0\leq t\leq T} |Y(t)|^{\gamma}\big]<\infty$
for every $\gamma>1$.
Combining \eqref{j-10}, \eqref{j-9} with \eqref{j-8}, we arrive at \eqref{j-7}.

Consequently, inserting \eqref{j-7} and \eqref{j-6} into \eqref{j-2}, we obtain
\begin{equation}\label{j-11}
\E[|M_T-\tilde M_T|^{q_0}]\leq C(\lambda,T, K_0,K_1,q_0) \de^{\frac{q_0}{2}\wedge(\alpha q_0)}.
\end{equation}

Next, we go to estimate the term $\E\big[|\la M\raa_T-\la \tilde M\raa_T|^{q_0}\big]$.
We have
\begin{align*}
  &\E\big[|\la M\raa_T-\la \tilde M\raa_T|^{q_0}\big]\\
  &=\E\Big[\Big|\int_0^T|\sigma^{-1}Z(s,Y(s))|^2-
  |\sigma^{-1}(Z_0(Y(s))-Z_0(Y(s_\de))-Z(s_\de,Y(s_\de)))|^2\d s\Big|^{q_0}\Big]\\
  &\leq \E\Big[\Big(\int_0^T\big(\big|\sigma^{-1}Z(s,Y(s))\big|
  +\big|\sigma^{-1}\big(Z_0(Y(s))-Z_0(Y(s_\de))-Z(s_\de,Y(s_\de))\big)\big|\big)\\
  &\quad\cdot
  \big||\sigma^{-1}(Z(s,Y(s))-Z(s_\de,Y(s_\de)))|+|\sigma^{-1}(Z_0(Y(s))-Z_0(Y(s_\de)))|\big|\d s\Big)^{q_0}\Big]\\
  &\leq C\E\big[\big(\la M\raa_T+\la \tilde M\raa_T\big)^{\frac{q_0\gamma}{2}}\big]^{\frac1{\gamma}}\\
  &\quad
  \cdot \E\Big[\Big(\int_0^T\!\big(|\sigma^{-1}(Z(s,Y(s))\!-\!Z(s_\de,Y(s_\de)))|
  \!+\!|\sigma^{-1}(Z_0(Y(s))\!-\!Z_0(Y(s_\de)))|\big)^{2}\d s\Big)^{\frac{q_0\gamma'}{2}}\Big]^{\frac{1}{\gamma'}},
\end{align*}
where $\gamma,\gamma'>1$, $1/\gamma+1/\gamma'=1$. Similar to the deduction of \eqref{M-4}, we have for every $\gamma>1$,
\[\E\big[\big(\la M\raa_T+\la \tilde M\raa_T\big)^{\frac{q_0\gamma}{2}}\big]<\infty.\]
Meanwhile, similar to the estimates \eqref{j-6} and \eqref{j-7}, we can obtain that
\begin{equation}\label{j-12}
\begin{split}
  &\E\Big[\Big(\int_0^T|\sigma^{-1}(Z_0(Y(s))-Z_0(Y(s_\de)))|^2\d s\Big)^{\frac{q_0\gamma'}{2}}\Big]\leq C\de^{\frac{q_0\gamma'}{2}},\\
  &\E\Big[\Big(\int_0^T|\sigma^{-1}(Z(s,Y(s))-Z(s_\de,Y(s_\de)))|^2\d s\Big)^{\frac{q_0\gamma'}{2}}
  \Big]\leq C\de^{\frac{q_0\gamma'}{2}\wedge (\alpha q_0\gamma')}.
\end{split}
\end{equation}
Therefore,
\begin{equation}\label{j-13}
 \E\big[|\la M\raa_T-\la \tilde M\raa_T|^{q_0}\big]\leq C\de^{\frac{q_0}{2}\wedge (\alpha q_0)}.
\end{equation}

In all, inserting the estimate \eqref{j-12} and \eqref{j-13} into \eqref{j-1}, we conclude that
\[\sup_{0\leq t\leq T} |\E f(X(t))-\E f(X_\de(t))|\leq C \de^{\frac 12\wedge \alpha},\]
which is the desired result. \fin

\subsection{Degenerate Case}
Our method to deal with EM's approximation for non-degenerate SDEs can be extended to deal with the degenerate system \eqref{deg-1}.
The basic idea is still to apply the Girsanov theorem to provide another representation of the solutions to SDEs \eqref{deg-1}, \eqref{deg-2}, and \eqref{deg-5} and use the Harnack inequality to verify Novikov's condition.

Now, we introduce the following auxiliary process:
\begin{equation}\label{edeg-1}
\begin{split}
  \d X(t)&=Y(t)\d t,\\
  \d Y(t)&=Z_0(Y(t))\d t+\sigma \d W(t),
\end{split}
\end{equation} where $Z_0$ is the vector corresponding to $V\in \mathscr V$ determined by \eqref{2.3}. Note that this process $(Y(t))$ is the same as the one used in non-degenerate case; however, the process $(X(t))$ is different.  Next, we rewrite \eqref{edeg-1} into three different forms to provide another representation of \eqref{deg-1}, \eqref{deg-2} and \eqref{deg-5} respectively. First, we have
\begin{equation}\label{edeg-2}
\begin{split}
  \d X(t)&= Y(t)\d t,\\
  \d Y(t)&= b(t,X(t), Y(t))\d t+\sigma \d \widehat W_3(t),
\end{split}
\end{equation}
where
\begin{equation}\label{edeg-3}
\widehat W_3(t)=W(t)-\int_0^t\sigma^{-1}Z(s,X(s),Y(s))\d s,\quad t>0.
\end{equation}
Setting
\[\q_3=\exp\Big[\int_0^T \la\sigma^{-1}Z(s,X(s),Y(s)),\d W(s)\raa-\frac 12\int_0^T|\sigma^{-1}Z(s,X(s),Y(s))|^2\d s\Big]\p,\]
if Novikov's condition
\begin{equation}\label{edeg-4}
\E\exp\Big[\int_0^T|\sigma^{-1}Z(s,X(s),Y(s))|^2\d s\Big]<\infty
\end{equation}
holds, $\q_3$ is a probability measure and $(\widehat W_3(t))_{t\in[0,T]}$ is a new Brownian motion under $\q_3$. Moreover, $(X(t),Y(t))$ will be a solution of SDE \eqref{deg-1} under the probability $\q_3$.

Second, let us rewrite \eqref{edeg-1} into the following form:
\begin{equation}\label{edeg-5}
\begin{split}
  \d X(t)&=Y(t)\d t,\\
  \d Y(t)&=\tilde b(t ,X(t),Y(t))\d t+\sigma \d \widehat W_4(t),
\end{split}
\end{equation}
where
\begin{equation}\label{edeg-6}
\widehat W_4(t)=W(t)-\int_0^t\sigma^{-1}\wt Z(s,X(s),Y(s))\d s,\quad t>0.
\end{equation}
If
\begin{equation}\label{edeg-7}
\E\exp\Big[\int_0^T |\sigma^{-1}\wt Z(s,X(s),Y(s))|^2\d s\Big]<\infty,
\end{equation}
then
\begin{equation}\label{edeg-8}
\q_4:=\exp\Big[\int_0^T \la \sigma^{-1}\wt Z(s,X(s),Y(s)),\d W(s)\raa-\frac 12\int_0^T|\sigma^{-1}\wt Z(s,X(s),Y(s))|^2\d s\Big]\p
\end{equation} is a probability measure, and further $(W_4(t))_{t\in[0,T]}$ is a new Brownian motion under $\q_4$. Furthermore, $(X(t),Y(t))$ is a solution to SDE \eqref{deg-2} under the probability measure $\q_4$.

At last, let us rewrite \eqref{edeg-1} into the form:
\begin{equation}\label{edeg-9}
\begin{split}
  \d X(t)&=Y(t)\d t,\\
  \d Y(t)&=b(t_\de,X(t_\de),Y(t_\de))\d t+\sigma \d \widehat W_5(t),
\end{split}
\end{equation}
where
\begin{equation}\label{edeg-10}
\widehat W_5(t)=W(t)+\int_0^t \sigma^{-1}\big(Z_0(Y(s))-b(s_\de, X(s_\de),Y(s_\de))\big)\d s,\quad t>0.
\end{equation}
If
\begin{equation}\label{edeg-11}
\E\exp\Big[\int_0^T\big|\sigma^{-1}\big(Z_0(Y(s))-b(s_\de, X(s_\de),Y(s_\de))\big)\big|^2\d s\Big]<\infty,
\end{equation}
then
\begin{equation}\label{edeg-12}
\begin{split}
\q_5&:=\exp\Big[\int_0^T \la \sigma^{-1}\big(Z_0(Y(s))-b(s_\de, X(s_\de),Y(s_\de))\big),\d W(s)\raa\\
&\qquad \qquad -\frac 12\int_0^T\big|\sigma^{-1}\big(Z_0(Y(s))-b(s_\de, X(s_\de),Y(s_\de))\big)\big|^2\d s\Big]\p
\end{split}
\end{equation} is a probability measure, and $(W_5(t))_{t\in[0,T]}$ is a new Brownian motion under $\q_5$. Thus, $(X(t),Y(t))$ is a solution to SDE \eqref{deg-5} under $\q_5$.

Due to the conditions (A1) and (A2), using Lemmas \ref{t-3.1} and \ref{t-3.2}, we can check that \eqref{edeg-4}, \eqref{edeg-7}, \eqref{edeg-11} hold. Invoking the previous representation of the corresponding degenerate SDEs, we can prove Theorems \ref{t-deg-1}, \ref{t-deg-2} and \ref{t-deg-3} by the same method as that of Theorems \ref{t-0}, \ref{t-1} and \ref{t-2} respectively. The details are omitted.

\end{document}